\documentclass[11pt,a4paper]{amsart}
\usepackage[utf8]{inputenc}
\usepackage[T1]{fontenc}
\usepackage[english]{babel}

\usepackage{sidecap}

\usepackage{verbatim}
\usepackage{lmodern}
\usepackage{amsmath}
\usepackage{amssymb} 
\usepackage{amsthm} 
\usepackage{thmtools}
\usepackage{enumitem}
\usepackage{graphicx}
\usepackage[hidelinks]{hyperref} 
\usepackage{mathtools}
\usepackage{indentfirst}
\usepackage{tabularx}
\usepackage{bbm}
\usepackage[capitalise]{cleveref}
\usepackage{stmaryrd}
\usepackage{tabularx}

\usepackage[left=1.2in,top=2cm,right=1.2in,bottom=2cm]{geometry}

\usepackage{tikz} 
\usepackage{tikz-cd} 
\newcommand{\btk}{\begin{tikzcd}}
\newcommand{\etk}{\end{tikzcd}}
\usetikzlibrary{arrows,calc,decorations.markings}
\usepackage{xparse}

\usepackage{blindtext}

\makeatletter
\newcommand{\@bbify}[1]{
  \ifcsname b#1\endcsname
  \message{WARNING: Overwriting b#1 with blackboard letter!}
  \fi
  \expandafter\edef\csname b#1\endcsname
  {\noexpand\ensuremath{\noexpand\mathbb #1}\noexpand\xspace}}
\newcommand{\@calify}[1]{
  \ifcsname c#1\endcsname
  \message{WARNING: Overwriting c#1 with calligraphic letter!}
  \fi 
  \expandafter\edef\csname c#1\endcsname
  {\noexpand\ensuremath{\noexpand\mathcal #1}\noexpand\xspace}}
\newcommand{\@bfify}[1]{
  \ifcsname bf#1\endcsname
  \message{WARNING: Overwriting c#1 with bold letter!}
  \fi
  \expandafter\edef\csname bf#1\endcsname
  {\noexpand\ensuremath{\noexpand\mathbf #1}\noexpand\xspace}}
\newcounter{@letter}\stepcounter{@letter}
\loop\@bbify{\Alph{@letter}}\@calify{\Alph{@letter}}\@bfify{\Alph{@letter}}
\ifnum\the@letter<26\stepcounter{@letter}\repeat
\makeatother

\newenvironment{tz}{\begin{center}\begin{tikzpicture}}{\end{tikzpicture}\end{center}}
\tikzstyle{d}=[double distance=.3ex]
\tikzstyle{w}=[preaction={draw=white,-,line width=5pt}]

\tikzset{%
node distance=1.5cm, la/.style={scale=0.8}, lasmall/.style={scale=0.75}, over/.style={auto=false,fill=white,inner sep=1.5pt, minimum size=0, outer sep=0},
    symbol/.style={%
        draw=none,
        every to/.append style={%
            edge node={node [sloped, allow upside down, auto=false]{$#1$}}},
            
    }, pro/.style={postaction={decorate,decoration={
        markings,
        mark=at position .5 with {\node at (0,0) {$\bullet$};}
      }},
      inner sep=.9ex,
      },
      prosmall/.style={postaction={decorate,decoration={
        markings,
        mark=at position .5 with {\node at (0,0) {$\scriptstyle \bullet$};}
      }},
      inner sep=.9ex,
      },
  n/.style={double equal sign distance, -implies}, t/.style={double distance=2.5pt, -implies, postaction={draw,-}},
}

\NewDocumentCommand{\punctuation}{ m m O{5pt} }{\node at ($(#1.east)-(0,#3)$) {#2};}

\newcommand{\Set}{\mathrm{Set}}
\newcommand{\sSet}{\mathrm{sSet}}

\newcommand{\an}{\ensuremath{\mathrm{An}}}
\newcommand{\nfib}{\ensuremath{\mathrm{NFib}}}

\newcommand{\id}{\mathrm{id}}
\newcommand{\op}{\mathrm{op}}

\newcommand{\Cat}{\mathrm{Cat}}
\newcommand{\Gpd}{\mathrm{Gpd}}
\newcommand{\Loc}{\cL\mathrm{oc}}

\newcounter{diagram}
\renewcommand{\thediagram}{\thesubsection}
\newenvironment{diagram}{
\setcounter{diagram}{\value{subsection}}\refstepcounter{subsection}\refstepcounter{diagram}
\begin{center}
\normalfont{(\thediagram)}\hfill\begin{tikzpicture}[baseline=(current bounding box.center)]}
    {\end{tikzpicture}\hfill\text{ }\end{center}}

\newlist{rome}{enumerate}{7}
\setlist[rome]{label=(\roman*)}
\newlist{numbered}{enumerate}{7}
\setlist[numbered]{label=(\arabic*)}
\setlist[numbered,1]{start=0}

\newtheorem{theorem}{Theorem}[section]
\newtheorem{cor}[theorem]{Corollary}
\newtheorem{prop}[theorem]{Proposition}
\newtheorem{lem}[theorem]{Lemma}
\declaretheorem[name=Theorem,numbered=yes]{theoremA}

\theoremstyle{definition}
\newtheorem{defn}[theorem]{Definition}

\newtheorem{notation}[theorem]{Notation}

\newtheorem{constr}[theorem]{Construction}

\theoremstyle{remark}
\newtheorem{rem}[theorem]{Remark}

\crefname{theorem}{Theorem}{Theorems}
\crefname{theoremA}{Theorem}{Theorems}
\crefname{cor}{Corollary}{Corollaries}
\crefname{prop}{Proposition}{Propositions}
\crefname{lem}{Lemma}{Lemmas}

\crefname{defn}{Definition}{Definitions}
\crefname{terminology}{Terminology}{Terminologies}
\crefname{ex}{Example}{Examples}
\crefname{notation}{Notation}{Notations}
\crefname{descr}{Description}{Descriptions}
\crefname{constr}{Construction}{Constructions}

\crefname{rem}{Remark}{Remarks}

\title{A model structure for Grothendieck fibrations}

\author[L.\ Moser]{Lyne Moser}
\address{Fakultät für Mathematik, Universität Regensburg, 93040 Regensburg, Germany}
\email{lyne.moser@ur.de}

\author[M.\ Sarazola]{Maru Sarazola}
\address{School of Mathematics, University of Minnesota, Minneapolis MN, 55415, USA}
\email{maru@umn.edu}

\begin{document}

\begin{abstract}

    We construct two model structures, whose fibrant objects capture the notions of \emph{discrete fibrations} and of \emph{Grothendieck fibrations} over a category $\mathcal{C}$. For the discrete case, we build a model structure on the slice $\mathrm{Cat}_{/\mathcal{C}}$, Quillen equivalent to the projective model structure on $[\mathcal{C}^{\mathrm{op}},\mathrm{Set}]$ via the classical category of elements construction. The cartesian case requires the use of markings, and we define a model structure on the slice $\mathrm{Cat}^+_{/\mathcal{C}}$, Quillen equivalent to the projective model structure on $[\mathcal{C}^{\mathrm{op}},\mathrm{Cat}]$ via a marked version of the Grothendieck construction.
    
    We further show that both of these model structures have the expected interactions with their $\infty$-counterparts; namely, with the contravariant model structure on $\mathrm{sSet}_{/ N\mathcal{C}}$ and with Lurie's cartesian model structure on $\mathrm{sSet}^+_{/ N\mathcal{C}}$.
\end{abstract}

\maketitle

\section{Introduction}\label{section:intro}

The Grothendieck construction \cite{Grothendieck} plays a central role in category theory: it allows us to freely move between the worlds of indexed categories and fibered categories, providing access to tools and results from both. 

Given a functor $F\colon \cC^\op \to \Cat$, its Grothendieck construction yields a \emph{Grothendieck fibration} over $\cC$; that is, a functor $P\colon \cP\to\cC$ with the property that, for every morphism $f\colon c\to Pa$ in $\cC$, there is a $P$-cartesian morphism $g\colon a\to a'$ in $\cP$ such that $Pg=f$. In fact, the Grothendieck fibrations that arise in this way satisfy a special property: they are \emph{split}; meaning, they admit a choice of cartesian lifts which is functorial. Then, the Grothendieck construction gives an equivalence of categories  
\[ \textstyle \int_\cC\colon [\cC^\op,\Cat]\to\mathrm{GrFib}^{\mathrm{split}}_{/\cC}\] between the category of functors $\cC^\op\to\Cat$ and natural transformations, and the category of split Grothendieck fibrations over $\cC$ and cartesian functors, i.e., functors that preserve cartesian morphisms (see, for instance, \cite[Proposition 10.1.11, Theorem 10.6.16]{johnsonyau}).

In practice, however, many of the Grothendieck fibrations one encounters are not split---such as the codomain fibration---and one finds the need to broaden the scope of the above correspondence. To obtain all Grothendieck fibrations, we must instead work with \emph{pseudo-functors} $\cC^\op\to\underline{\Cat}$, where $\underline{\Cat}$ is now considered as a 2-category. We then get an equivalence of categories \[ \textstyle\int_\cC\colon \mathrm{Ps}[\cC^\op,\underline{\Cat}]\to\mathrm{GrFib}_{/\cC}\]
between the category of pseudo-functors $\cC^\op\to\Cat$ and pseudo-natural transformations, and the category of Grothendieck fibrations over $\cC$ and cartesian functors (see \cite[Theorem 10.6.16]{johnsonyau}). Of course, the passage from strict to pseudo-functors increases the complexity, as one must keep track of 2-dimensional coherence data. 

In this paper, we show how homotopy theory provides a new perspective on the Grothendieck construction. Instead of using pseudo-functors to recover all Grothendieck fibrations, our approach is to remain in the 1-categorical setting of strict functors; then, the required higher-dimensional information is encoded homotopically. 

Concretely, our goal is to construct a model structure on the slice category $\Cat_{/\cC}$ whose fibrant objects are the Grothendieck fibrations, using methods recently developed in  \cite{GMSV}. Unfortunately, this endeavor faces an immediate technical obstacle. Recall that 
a functor $P\colon \cP\to\cC$ is a Grothendieck fibration if, for every morphism $f\colon c\to Pa$ in $\cC$, there is a $P$-cartesian lift $g\colon a'\to a$ of~$f$. 
In turn, the morphism $g$ is \emph{$P$-cartesian}  if, for every morphism $h\colon b\to a$ in $\cP$ and every morphism $f\colon Pb\to Pa'$ in $\cC$ such that $Ph=(Pg)f$, there is a unique morphism $h'\colon b\to a'$ in $\cP$ with $Ph=f$ and $h=gh'$. While it is possible to use a lifting property to express what it means for a given morphism to be $P$-cartesian, we cannot encode the data of a Grothendieck fibration as a lifting property of the functor $P$ itself, as this would require us to have access to the set of all $P$-cartesian morphisms.

We choose to bypass this technical problem with a technical solution: by considering \emph{marked categories} instead, where the added decorations allow us to effectively ``mark'' the cartesian morphisms, and \emph{marked functors} between them, which preserve the markings and therefore allow us to encode cartesian functors. This is the strategy used by Lurie in the $\infty$-setting (see \cite[Section 3]{HTT}). Paraphrasing Lurie to fit our setting:
\begin{quote} In order to have a model category, we need to be able to form fibrant replacements: in
other words, we need the ability to enlarge an arbitrary functor $P\colon\cP\to\cC$ into
a commutative diagram
\begin{tz}
    \node[](1) {$\cP$}; 
    \node[below right of=1](3) {$\cC$};
    \node[above right of=3](2) {$\cQ$}; 
    \draw[->] (1) to node[above,la]{$\phi$} (2);
    \draw[->] (1) to node[left,la,yshift=-2pt]{$P$} (3); 
    \draw[->] (2) to node[right,la,yshift=-2pt]{$Q$} (3);
\end{tz}
where $Q$ is a Grothendieck fibration built from $P$. A question arises: for which morphisms $f$ of $\cP$ should $\phi(f)$ be $Q$-cartesian? This sort of information
is needed for the construction of $\cQ$; consequently, we need a formalism in which certain morphisms of $\cC$ have been distinguished.
\end{quote}

With this tool in hand, the notion of Grothendieck fibration $P\colon\cP\to\cC$ is then upgraded to that of a \emph{cart-marked Grothendieck fibration}: a functor between marked categories $P\colon (\cP,E)\to (\cC, \mathrm{Mor}\cC)$ where the set $E$ of marked morphisms in $\cP$ consists precisely of the $P$-cartesian morphisms. Letting $\Cat^+$ denote the category of marked categories and functors that preserve the markings, our first main result reads as follows.

\begin{theoremA} \label{thmA}
There is a model structure on $\Cat^+_{/\cC}$ in which 
\begin{rome}[leftmargin=1.1cm]
\item the cofibrations are the injective-on-objects functors,
\item the fibrant objects are the cart-marked Grothendieck fibrations,
\item the weak equivalences between fibrant objects are the equivalences of categories, 
\item the fibrations between fibrant objects are the isofibrations.
\end{rome}
\end{theoremA} 

The classical Grothendieck construction naturally admits a marked analogue. We introduce it in \cref{section:markedgroth}, and show the following result, which is notably a 1-categorical (albeit homotopical) statement, as it involves strict (and not pseudo-) functors.

\begin{theoremA}\label{thmB} 
The marked Grothendieck construction $\int^+_\cC$ is part of an adjunction
\begin{tz}
\node[](1) {$[\cC^{\op},\Cat]$}; 
\node[right of=1,xshift=1.6cm](2) {$\Cat^+_{/\cC}$}; 

\draw[->] ($(1.east)-(0,5pt)$) to node[below,la]{$\int^+_\cC$} ($(2.west)-(0,5pt)$);
\draw[->] ($(2.west)+(0,5pt)$) to node[above,la]{$\cT^+_\cC$} ($(1.east)+(0,5pt)$);

\node[la] at ($(1.east)!0.5!(2.west)$) {$\bot$};
\end{tz}
which is a Quillen equivalence between the projective model structure and the model structure of \cref{thmA}.
\end{theoremA}
 
The Grothendieck construction plays an even more foundational role in $(\infty,1)$-category theory where it goes by the name ``straightening-unstraightening'' \cite{HTT}. In practice, it is infamously hard to construct a homotopy coherent diagram from an $\infty$-category $\cC$ to the $\infty$-category of $\infty$-categories. The fibrational approach provides a more approachable alternative: it allows us to instead do this by constructing a map of simplicial sets over that same $\infty$-category $\cC$. 

We show that our structure is homotopically compatible with Lurie's construction in the $\infty$-setting, by proving the following.

\begin{theoremA}
    The following square of homotopically fully faithful right Quillen functors commutes up to a level-wise weak equivalence. 
    \begin{tz}
\node[](1) {$[\cC^{\op},\Cat]_{\mathrm{proj}}$}; 
\node[right of=1,xshift=2.8cm](2) {$(\Cat^+_{/\cC^\sharp})_{\mathrm{GrFib}}$}; 
\node[below of=1](1') {$[\cC^\op,\sSet^+]_\mathrm{proj}$};
\node[below of=2](2') {$(\sSet^+_{/(N\cC)^\sharp})_{\mathrm{Cart}}$}; 

\draw[->] (1) to node[above,la]{$\int^+_\cC$} node[below,la]{$\simeq_{QE}$} (2); 
\draw[right hook->] (1) to node[left,la]{$(N^+(-)^\natural)_*$} (1'); 
\draw[right hook->] (2) to node[right,la]{$N^+$} (2');
\draw[->] (1') to node[below,la]{$\mathrm{Un}_\cC^+$} node[above,la]{$\simeq_{QE}$} (2');
\end{tz}
\end{theoremA}

In particular, this means that \cref{thmA,thmB} should be interpreted as the 1-categorical versions of Lurie's cartesian model structure on marked simplicial sets and his straightening-unstraightening construction. Notably, they have the benefit that they are not obtained as a restriction of the existing $\infty$-machinery, but are developed entirely 1-categorically.

One can also consider a discrete version of the Grothendieck construction 
\[ \textstyle \int_\cC\colon [\cC^\op,\Set]\to\Cat_{/\cC}\]
where inputs are now functors to $\Set$. This construction---
called the \emph{category of elements}---gives an equivalence between the categories of functors from $\cC^\op$ to $\Set$ and of discrete fibrations over $\cC$ \cite[Proposition 1.1.7]{elephant}, and is intimately linked with the study of representable functors and universal properties. Our results above have analogous versions in this discrete setting.

These homotopical results in the discrete case are less exciting, as the adjunction $\cT_\cC\dashv\int_\cC$ is known to be an equivalence of categories when restricted to the category of discrete fibrations over $\cC$, and hence there is no homotopical strictification occurring here. However, understanding the techniques used in the discrete case is a useful step towards understanding the more complex case of Grothendieck fibrations, so we discuss the discrete case first for pedagogical reasons.

\subsection*{Outline} 

\cref{section:discrete} recalls the definition of discrete fibrations, and constructs a model structure on $\Cat_{/\cC}$ whose fibrant objects are the discrete fibrations. \cref{section:catofelements} recalls the category of elements functor and its adjoint, and shows that they are Quillen equivalences. \cref{section:cartesian,section:markedgroth} establish the general (non-discrete) case: in \cref{section:cartesian} we discuss marked categories, (cart-marked) Grothendieck fibrations, and construct a model structure on $\Cat^+_{/\cC}$ whose fibrant objects are the cart-marked Grothendieck fibrations. In \cref{section:markedgroth} we introduce a marked Grothendieck construction, and prove it is a Quillen equivalence. Finally, \cref{section:discinfty,section:cartinfty} compare our results to their infinity analogues in the literature. 

\subsection*{Acknowledgments}

We would like to thank Tim Campion, Kim Nguyen, Nima Rasekh and Jonathan Weinberger for interesting discussions related
to the subject of this paper. We would also like to thank an anonymous referee for suggesting various improvements. During the realization of this work, the first named author was a member of the Collaborative Research Centre ``SFB 1085: Higher
Invariants'' funded by the Deutsche Forschungsgemeinschaft (DFG).

\section{Model structure for discrete fibrations}\label{section:discrete}

Throughout the paper we fix a category $\cC$. In this section, we construct a model structure on the category $\Cat_{/\cC}$ of categories over $\cC$ whose fibrant objects are the discrete fibrations. Let us first recall some terminology.

\begin{defn}
Let $P\colon \cP\to \cC$ be a functor. Given a morphism $f\colon c'\to c$ in $\cC$, a morphism $g\colon a'\to a$ in $\cP$ such that $Pg=f$ is called a \textbf{$P$-lift} of $f$. 

A functor $P\colon \cP\to \cC$ is a \textbf{discrete fibration} if, for every object $a\in \cP$ and every morphism $f\colon c\to Pa$ in $\cC$, there is a unique $P$-lift $g\colon a'\to a$ of $f$. 
\end{defn}

\begin{notation}
    We denote by $\Cat_{/\cC}$ the slice category over $\cC$; that is, the category whose
    \begin{numbered}[leftmargin=1.1cm]
        \item objects are functors $P\colon \cP\to\cC$,
        \item morphisms $(P\colon\cP\to\cC)\to (Q\colon \cQ\to \cC)$ are functors $F\colon \cP\to\cQ$ with $QF=P$.
    \end{numbered}
\end{notation}

The desired model structure on $\Cat_{/\cC}$ will be related to the trivial model structure on~$\Cat$, which we now recall. 

\begin{theorem} \label{thm:MStrivialCat}
There is a combinatorial model structure on the category $\Cat$ , called the \emph{trivial model structure}, in which 
\begin{rome}[leftmargin=1.1cm]
\item the weak equivalences are the isomorphisms of categories, 
\item the cofibrations are all functors,
\item the fibrations are all functors. 
\end{rome}
\end{theorem}

From the trivial model structure on $\Cat$, we obtain a model structure on the slice category $\Cat_{/\cC}$ created by the forgetful functor $\Cat_{/\cC}\to \Cat$. Note that this model structure does not capture the desired class of fibrant objects as every functor $\cP\to \cC$ in $\Cat_{/\cC}$ is fibrant. We aim to construct a localization of the slice model structure on $\Cat_{/\cC}$ in which the fibrant objects are the discrete fibrations.

Our strategy is to apply a new technique for constructing model structures due to the authors, Guetta, and Verdugo \cite[Theorem 2.8]{GMSV}. This result requires us to specify a set of generating cofibrations, an auxiliary weak factorization system $(\an,\nfib)$ (whose role is to detect the fibrant objects and the fibrations between them) such that all morphisms in $\an$ are cofibrations, and a class $\cW_f$ of weak equivalences between fibrant objects. We now introduce these classes in our setting, and verify that the conditions of the theorem are satisfied.

\begin{defn}
The class $\an$ of \textbf{anodyne extensions} is the smallest weakly saturated class of functors over~$\cC$ which contains the functors 
\[ [0]\xrightarrow{1} [1]\xrightarrow{f} \cC \quad \text{and} \quad [1]\sqcup_{[0]} [1]\to [1]\xrightarrow{f} \cC \]
for all morphisms $f$ in $\cC$, where $[1]\sqcup_{[0]} [1]$ is the pushout of the span $[1]\xleftarrow{1}[0]\xrightarrow{1} [1]$.
\end{defn}

\begin{defn}
    The class $\nfib$ of \textbf{naive fibrations} consists of the morphisms in $\Cat_{/\cC}$ which have the right lifting property with respect to all morphisms in $\an$. Then an object in $\Cat_{/\cC}$ is called \textbf{naive fibrant} if the unique morphism to the terminal object $\id_\cC$ is a naive fibration.  
\end{defn}

\begin{rem}\label{fibobj}
By unpacking the lifting conditions, we get the following characterizations.
\begin{rome}[leftmargin=1.1cm]
    \item A functor $\cP\to \cQ$ over $\cC$ is a naive fibration if and only if it is a discrete fibration. 
    \item A functor $\cP\to \cC$ is naive fibrant if and only if it is a discrete fibration. 
\end{rome}
\end{rem}

The following fact regarding discrete fibrations, which can be deduced directly from the definition, will be useful for our purposes.

\begin{lem} \label{prop:discfib}
Given composable functors $F\colon \cP\to \cQ$ and $Q\colon \cQ\to \cC$ such that $Q$ is a discrete fibration, then $F$ is a discrete fibration if and only if $QF$ is so. In particular, a functor $\cP\to \cQ$ between discrete fibrations over $\cC$ is a discrete fibration.
\end{lem}

\begin{rem}
    We can use the weak factorization system $(\an,\nfib)$ to factor any functor $\cP\to \cC$ over $\cC$ as an anodyne extension $J\colon \cP\to \widehat{\cP}$ followed by a naive fibration $\widehat{\cP} \to \cC$; in particular, this makes $\widehat{\cP} \to\cC$ a discrete fibration. From now on, we refer to a map $\cP\to \widehat{\cP}$ as above as a \textbf{naive fibrant replacement}.

    The pair $(\an,\nfib)$ that we use here is known in the literature as the \textit{comprehensive factorization system} \cite{final}, and the maps in $\an$ are precisely the final functors. In fact, this is an orthogonal factorization system, and so the naive fibrant replacements mentioned above are unique up to unique isomorphism. We mention this for the curious reader, although none of these facts are used in our results. 
\end{rem}

Since we expect this new model structure to be a localization of the slice model structure, we let the cofibrations be the same, namely all functors over $\cC$. In particular, the trivial fibrations are the isomorphisms over $\cC$. Similarly, we let the class of weak equivalences between naive fibrant objects be given by the functors $\cP\to \cQ$ over $\cC$ that are isomorphisms of categories.

Unfortunately, weak equivalences between arbitrary objects do not admit a similarly convenient description. According to \cite[Definition 2.4]{GMSV}, they are defined through a naive fibrant replacement.

\begin{defn}\label{defn:we}
    A functor $F\colon \cP\to\cQ$ over $\cC$ is a \textbf{weak equivalence} if there exists a naive fibrant replacement of $F$ which is an isomorphism. More explicitly, if there exists a commutative diagram of functors over $\cC$
    \begin{tz}
\node[](1) {$\cP$}; 
\node[right of=1](2) {$\cQ$}; 
\node[below of=1](1') {$\widehat{\cP}$}; 
\node[below of=2](2') {$\widehat{\cQ}$}; 

\draw[->] (1) to node[above,la]{$F$} (2); 
\draw[->] (1) to node[left,la]{$\iota_\cP$} (1'); 
\draw[->] (2) to node[right,la]{$\iota_\cQ$} (2'); 
\draw[->] (1') to node[below,la]{$\widehat{F}$} (2');
\end{tz}
    where $\widehat{\cP}\to\cC$, $\widehat{\cQ}\to\cC$ are discrete fibrations, $\iota_\cP, \iota_\cQ$ are anodyne extensions, and $\widehat{F}$ is an isomorphism.  
\end{defn}

Now that all of the relevant classes of maps have been introduced, we can prove the existence of the desired model structure on  $\Cat_{/\cC}$ for discrete fibrations.

\begin{theorem}\label{thm:discreteMS}
There is a combinatorial model structure on the category $\Cat_{/\cC}$, denoted by $(\Cat_{/\cC})_{\mathrm{discfib}}$, in which 
\begin{rome}[leftmargin=1.1cm]
\item the cofibrations are all functors, 
\item the fibrant objects are the discrete fibrations, 
\item the weak equivalences between fibrant objects are the isomorphisms, 
\item every functor between fibrant objects is a fibration.
\end{rome}
\end{theorem}

\begin{proof}
We apply \cite[Theorem 2.8]{GMSV} using \cite[Proposition 2.21]{GMSV}. First note that the class of cofibrations is indeed cofibrantly generated by the set of functors over $\cC$ containing
\[ \emptyset\to [0]\xrightarrow{c} \cC, \quad [0]\sqcup [0]\to [0]\xrightarrow{c} \cC, \quad [0]\sqcup [0]\to [1]\xrightarrow{f} \cC, \quad [1]\sqcup_{[0]\sqcup[0]} [1]\to [1]\xrightarrow{f} \cC \]
for all objects $c$ in $\cC$ and all morphisms $f$ in $\cC$. 

Condition~(1) (trivial fibrations are weak equivalences) is clear, as trivial fibrations are precisely the isomorphisms. Condition (2) (2-out-of-6 for $\cW_f$) is clear as the weak equivalences between fibrant objects are defined to be isomorphisms. Condition (3) (accessibility of $\cW_f$) follows from the fact that the class of isomorphisms of categories is accessible. 

To check condition (P) (path objects for fibrant objects), note that for any discrete fibration $\cP\to \cC$, the factorization 
\[ \cP\xrightarrow{\id_\cP}\cP\to \cP\times_{\cC} \cP \]
of the diagonal functor at $\cP$ over $\cC$ is a path object. Indeed, the first functor is an isomorphism, and the second functor is a naive fibration by \cref{prop:discfib}.

Finally, condition (5) (naive fibrations that are weak equivalences between fibrant objects are trivial fibrations) is clear.

This proves that there exists a model structure on $\Cat_{/\cC}$ with cofibrations and fibrant objects as described in the statement. The description in item (iii) is by definition, and the one in item (iv) can be deduced from \cref{prop:discfib}.
\end{proof}

\section{The category of elements}\label{section:catofelements}

We now consider the category of elements of a functor $F\colon \cC^\op\to\Set$. This construction gives a functor 
\[ \textstyle \int_\cC\colon [\cC^{\op},\Set]\to \Cat_{/\cC} \]
and in this section we show this is in fact a Quillen equivalence between the projective model structure on $[\cC^{\op},\Set]$ and the model structure $(\Cat_{/\cC})_{\mathrm{discfib}}$ for discrete fibrations constructed in \cref{thm:discreteMS}. 

 We begin by recalling the definition of the category of elements functor.

\begin{defn}
We define a functor $\int_\cC\colon [\cC^{\op},\Set]\to \Cat_{/\cC}$ as follows.
\begin{numbered}[leftmargin=1.1cm]
\item Given a functor $F\colon \cC^{\op}\to \Set$, its \textbf{category of elements} $\int_\cC F$ is the category whose 
\begin{numbered}[start=0]
\item objects are pairs $(c,x)$ of an object $c\in \cC$ and an element $x\in Fc$, 
\item morphisms $(c,x)\to (d,y)$ are morphisms $f\colon c\to d$ in $\cC$ such that $x=Ff(y)$.
\end{numbered}
This comes with a canonical projection $\int_\cC F\to \cC$ onto the first coordinate. 

\item Given a natural transformation $\alpha\colon F\Rightarrow G$ between functors $F,G\colon \cC^{\op}\to \Set$, the induced functor $\int_\cC\alpha\colon \int_\cC F\to \int_\cC G$ over $\cC$ sends
\begin{numbered}[start=0]
\item an object $(c,x)$ to the object $(c,\alpha_c(x))$, 
\item a morphism $f\colon (c,x)\to (d,y)$ to the morphism $f\colon (c,\alpha_c(x))\to (d,\alpha_d(y))$, where $Gf(\alpha_d(y))=\alpha_c(Ff(y))=\alpha_c(x)$.
\end{numbered}
\end{numbered}
\end{defn}

\begin{rem}\label{grothisdiscfib}
Note that for any functor $F\colon \cC^{\op}\to \Set$, the canonical projection $\int_\cC F\to \cC$ is a discrete fibration. 
\end{rem}

The category of elements functor admits a left adjoint, which we now recall. 

\begin{notation} \label{commaoverP}
Given a functor $P\colon \cP\to \cC$ and an object $c\in \cC$, the comma category $c\downarrow P$ is the category whose 
\begin{numbered}[leftmargin=1.1cm]
\item objects are pairs $(a,f)$ of an object $a\in \cP$ and a morphism $f\colon c\to Pa$ in $\cC$, 
\item morphisms $(a',f')\to (a,f)$ are morphisms $g\colon a'\to a$ in $\cP$ such that $f=(Pg)f'$.
\end{numbered}

Note that a morphism $c\to d$ in $\cC$ induces a functor $d\downarrow P\to c\downarrow P$ by pre-composition. Moreover, a functor $\cP\to \cQ$ over $\cC$ between functors $P\colon \cP\to \cC$ and $Q\colon \cQ\to \cC$ induces, for every object $c\in \cC$, a functor $c\downarrow P\to c\downarrow Q$ using the fact that $P=QF$.
\end{notation}

\begin{notation}
We write $\pi_0\colon \Cat\to \Set$ for the left adjoint of the inclusion $\Set\hookrightarrow \Cat$. 
\end{notation}

\begin{defn}
We define a functor $\cT_\cC\colon \Cat_{/\cC}\to [\cC^{\op},\Set]$ as follows.
\begin{numbered}[leftmargin=1.1cm]
    \item Given a functor $P\colon \cP\to \cC$, the functor $\cT_\cC P\colon \cC^{\op}\to \Set$ sends
\begin{numbered}[start=0]
\item an object $c\in \cC$ to the set $\pi_0(c\downarrow P)$,
\item a morphism $c\to d$ in $\cC$ to the induced map $\pi_0(d\downarrow P)\to \pi_0(c\downarrow P)$.
\end{numbered}

\item Given a functor $F\colon \cP\to \cQ$ over $\cC$, the natural transformation $\cT_\cC F\colon \cT_\cC P\Rightarrow \cT_\cC Q$ is the one whose component at an object $c\in \cC$ is given by the induced map $\pi_0(c\downarrow P)\to \pi_0(c\downarrow Q)$. 
\end{numbered}
\end{defn}

The following is a combination of \cite[Proposition 1.10 and Lemmas 1.11 and 1.16]{nima}.

\begin{prop} \label{adjclassicalGC}
There is an adjunction
\begin{tz}
\node[](1) {$[\cC^{\op},\Set]$}; 
\node[right of=1,xshift=1.5cm](2) {$\Cat_{/\cC}$}; 

\draw[->] ($(1.east)-(0,5pt)$) to node[below,la]{$\int_\cC$} ($(2.west)-(0,5pt)$);
\draw[->] ($(2.west)+(0,5pt)$) to node[above,la]{$\cT_\cC$} ($(1.east)+(0,5pt)$);

\node[la] at ($(1.east)!0.5!(2.west)$) {$\bot$};
\end{tz}
such that the right adjoint $\int_\cC$ is fully faithful with essential image the discrete fibrations over $\cC$.
\end{prop}

In order to show that this is a Quillen equivalence, we first recall the definition of the projective model structure on the category of functors $[\cC^\op,\Set]$, which is constructed from the trivial model structure on $\Set$. Note that it also coincides with the trivial model structure on $[\cC^{\op},\Set]$. 

\begin{theorem}\label{projMSset}
    There is a model structure on $[\cC^{\op},\Set]$, called the projective model structure and denoted by $[\cC^{\op},\Set]_{\mathrm{proj}}$, in which
    \begin{rome}[leftmargin=1.1cm]
\item the weak equivalences are the level-wise isomorphisms of sets, 
\item the cofibrations are all maps,
\item the fibrations are all maps. 
\end{rome}
\end{theorem}

\begin{theorem}\label{thm:discreteQuillenequiv}
The adjunction
\begin{tz}
\node[](1) {$[\cC^{\op},\Set]_{\mathrm{proj}}$}; 
\node[right of=1,xshift=2.4cm](2) {$(\Cat_{/\cC})_{\mathrm{discfib}}$}; 

\draw[->] ($(1.east)-(0,5pt)$) to node[below,la]{$\int_\cC$} ($(2.west)-(0,5pt)$);
\draw[->] ($(2.west)+(0,5pt)$) to node[above,la]{$\cT_\cC$} ($(1.east)+(0,5pt)$);

\node[la] at ($(1.east)!0.5!(2.west)$) {$\bot$};
\end{tz}
is a Quillen equivalence.
\end{theorem}

\begin{proof}
To show that $\int_\cC$ is right Quillen, we must show that it preserves fibrations and trivial fibrations. As mentioned in \cref{grothisdiscfib}, the functor $\int_\cC$ sends every object in $[\cC^{\op},\Set]$ to a discrete fibration; hence it preserves fibrations by \cref{prop:discfib}, and trivial fibrations as every functor preserves isomorphisms.

As the functor $\int_\cC$ is fully faithful, the (derived) counit is an isomorphism. It remains to show that the (derived) unit is a weak equivalence. First, note that the component of the unit at a discrete fibration $\cP\to \cC$ is an isomorphism, as $\cP\to \cC$ is in the essential image of $\int_\cC$. Now, given any functor $P\colon \cP\to \cC$, consider a fibrant replacement $\cP\to \widehat{\cP}$ in $(\Cat_{/\cC})_{\mathrm{discfib}}$ with target $\widehat{P}\colon \widehat{\cP}\to \cC$. This gives a commutative square of functors over $\cC$,
\begin{tz}
\node[](1) {$\cP$}; 
\node[right of=1,xshift=.6cm](2) {$\int_\cC \cT_\cC(P)$}; 
\node[below of=1](1') {$\widehat{\cP}$};
\node[below of=2](2') {$\int_\cC \cT_\cC(\widehat{P})$}; 

\draw[->] (1) to (2); 
\draw[->] (1) to node[below,la, sloped]{$\sim$} (1'); 
\draw[->] (2) to node[above,la, sloped,yshift=-2pt]{$\sim$} (2');
\draw[->] (1') to node[below,la]{$\cong$} (2');
\end{tz}
where the bottom functor is an isomorphism since $\widehat{P}\colon \widehat{\cP}\to \cC$ is a discrete fibration, and the left-hand functor is a weak equivalence by construction. Moreover, the right-hand functor is also a weak equivalence. Indeed, the fact that all objects in $(\Cat_{/\cC})_{\mathrm{discfib}}$ are cofibrant and all objects in $[\cC^{\op},\Set]_{\mathrm{proj}}$ are fibrant ensures that both $\cT_\cC$ and $\int_\cC$ preserve weak equivalences by Ken Brown's lemma \cite[Lemma 1.1.12]{hovey}. Hence by $2$-out-of-$3$, the component of the unit at $P\colon \cP\to \cC$ is also a weak equivalence, as desired. 
\end{proof}

\section{Model structure for Grothendieck fibrations}\label{section:cartesian}

The aim of this section is to prove the analogous version of \cref{thm:discreteMS}, where discrete fibrations are replaced by Grothendieck fibrations. Let us start by recalling the relevant definitions, as well as some useful properties. We refer the reader to \cite[\S 3.1.1]{vistoli} for more details.

\begin{defn}
    Let $P\colon \cP\to \cC$ be a functor. A morphism $g\colon a'\to a$ in $\cP$ is \textbf{$P$-cartesian} if, for every morphism $h\colon b\to a$ in $\cP$ and every morphism $f\colon Pb\to Pa'$ in $\cC$ such that $Ph=(Pg)f$, there is a unique $P$-lift $h'\colon b\to a'$ of~$f$ such that $h=gh'$. 

    A functor $P\colon \cP\to \cC$ is a \textbf{Grothendieck fibration} if, for every object $a\in \cP$ and every morphism $f\colon b\to Pa$ in $\cC$, there is a $P$-cartesian lift $g\colon a'\to a$ of $f$; i.e., a $P$-lift of~$f$ that is $P$-cartesian. 
\end{defn}

\begin{prop} \label{propPcart}
For any functor $P\colon \cP\to \cC$, the following hold. 
\begin{rome}[leftmargin=1.1cm] 
\item The composite of two $P$-cartesian morphisms is $P$-cartesian.
\item A morphism $g$ in $\cP$ is an isomorphism if and only if it is $P$-cartesian and $Pg$ is an isomorphism in $\cC$. 
    \item If $g\colon a'\to a$ and $g'\colon a''\to a$ are two $P$-cartesian lifts of the same morphism $Pg=Pg'$, there is a (unique) isomorphism $h\colon a''\xrightarrow{\cong} a'$ that is a $P$-lift of the identity at $Pa''=Pa'$ such that $g'=g h$. 
\end{rome}
\end{prop}

\begin{lem} \label{trivfibandcartmor}
Let $F\colon \cP\to \cQ$ be a trivial fibration in $\Cat_{/\cC}$ between functors $P\colon \cP\to \cC$ and $Q\colon \cQ\to \cC$. Then a morphism $g$ in $\cP$ is $P$-cartesian if and only if $Fg$ in $\cQ$ is $Q$-cartesian.
\end{lem}

\begin{proof}
This follows from the fact that, if $F$ is a trivial fibration, then it is surjective on objects and fully faithful on morphisms by \cref{chartrivfib}.
\end{proof}

The desired model structure for Grothendieck fibrations will be related to the canonical model structure on~$\Cat$ (often referred to as ``categorical'' or ``folk''). The existence of this model structure, which is nowadays widely known, is due to Joyal and Tierney; the reader may find the details in a note by Rezk \cite{rezk}. We first recall the definitions of the relevant classes of morphisms.

\begin{defn}
A functor $F\colon \cC\to \cD$ is
\begin{rome}[leftmargin=1.1cm]
\item an \textbf{equivalence of categories} if there is a functor $G\colon \cD\to \cC$ and natural isomorphisms $\id_\cC\cong GF$ and $FG\cong \id_\cD$,
\item an \textbf{isofibration} if, for every object $c\in \cC$ and every isomorphism $g\colon d\xrightarrow{\cong} Fc$ in $\cD$, there is an isomorphism $f\colon c'\xrightarrow{\cong} c$ in $\cC$ such that $Ff=g$. 
\end{rome}
\end{defn}

\begin{theorem} \label{thm:MSCat}
There is a combinatorial model structure on the category $\Cat$, called the \emph{canonical model structure}, in which
\begin{rome}[leftmargin=1.1cm]
\item the weak equivalences are the equivalences of categories, 
\item the cofibrations are the injective-on-objects functors,
\item the fibrations are the isofibrations. 
\end{rome}
\end{theorem}

By unpacking the definitions, one can explicitly describe the trivial fibrations as follows.

\begin{lem} \label{chartrivfib}
The functor $F\colon \cC\to \cD$ is a trivial fibration if and only if it is surjective on objects and fully faithful on morphisms.
\end{lem}

As explained in the introduction, Grothendieck fibrations over $\cC$ cannot be directly encoded by a lifting condition in the category $\Cat_{/\cC}$. Instead, we use \emph{marked categories} so we are able to remember which morphisms are the $P$-cartesian morphisms corresponding to a given functor $P\colon\cP\to\cC$.

\begin{defn}
A \textbf{marked category} is a pair $(\cP,E)$ of a category $\cP$ together with a subset $E$ of the set of morphisms of $\cP$ containing the identities. A morphism in $E$ is called \textbf{marked}.

A \textbf{marked functor} $(\cP,E_\cP)\to (\cQ,E_\cQ)$ is a functor $F\colon \cP\to \cQ$ that preserves the marking, i.e., such that, for every $f\in E_\cP$, we have $Ff\in E_\cQ$. 

We denote by $\Cat^+$ the category of marked categories and marked functors. 
\end{defn}

\begin{rem}\label{forgetmarkings}
The forgetful functor $U\colon \Cat^+\to \Cat$ has both a left and a right adjoint. The left adjoint $(-)^\flat\colon \Cat\to \Cat^+$ endows a category with the minimal marking; namely, only identities are marked. The right adjoint $(-)^\sharp\colon \Cat\to \Cat^+$ endows a category with the maximal marking; that is, all morphisms are marked. 
\end{rem}

\begin{defn}
    We denote by $\Cat^+_{/\cC^\sharp}$ the slice category over $\cC^\sharp$; that is, the category whose
    \begin{numbered}[leftmargin=1.1cm]
        \item objects are marked functors $(\cP,E)\to\cC^\sharp$,
        \item morphisms $(P\colon (\cP,E_\cP)\to\cC^\sharp)\to (Q\colon (\cQ,E_\cQ)\to \cC^\sharp)$ consist of marked functors $F\colon (\cP,E_\cP)\to(\cQ,E_\cQ)$ such that $QF=P$.
    \end{numbered}
    Note that a marked functor $(\cP,\cE)\to \cC^\sharp$ is simply a functor $\cP\to \cC$, as all morphisms in~$\cC^\sharp$ are marked.
\end{defn}

Our goal is to construct a model structure on $\Cat^+_{/\cC^\sharp}$ in which the fibrant objects $P\colon (\cP,E)\to \cC^\sharp$ are the Grothendieck fibrations $P\colon \cP\to \cC$ with marking $E$ given by the set of $P$-cartesian morphisms. We call such a functor a \textbf{cart-marked Grothendieck fibration}. Once again, our strategy is to apply \cite[Theorem 2.8]{GMSV}, and to this end, we now introduce the required classes of morphisms.

\begin{notation}
    We use $\bI$ to denote the free-living isomorphism, and $\Lambda^2[2]\hookrightarrow [2]$ to denote the following inclusion of categories.
    \begin{tz}
        \node[](1) {$1$}; 
        \node[below left of=1](0) {$0$}; 
        \node[below right of=1](2) {$2$}; 
        \draw[->] (0) to (2); 
        \draw[->] (1) to (2);

        \node[right of=1,xshift=3cm](1') {$1$}; 
        \node[below left of=1'](0) {$0$}; 
        \node[below right of=1'](2) {$2$}; 
        \draw[->] (0) to (2); 
        \draw[->] (1') to (2);
        \draw[->] (0) to (1'); 
        \node at ($(0)!0.5!(2)+(0,.4cm)$) {\rotatebox{90}{$=$}};
        \node at ($(1)!0.5!(1')+(0,-.6cm)$) {$\hookrightarrow$};
    \end{tz}
\end{notation}

\begin{defn}
The class $\an^+$ of \textbf{marked anodyne extensions} is the smallest weakly saturated class of marked functors over $\cC^\sharp$ which contains the following marked functors:
\begin{rome}[leftmargin=1.1cm]
\item the inclusion $[0]\xrightarrow{1} [1]^\sharp\to \cC^\sharp$,  
\item the inclusion $(\Lambda^2[2],\{ 1\to 2\})\to ([2],\{ 1\to 2\})\to \cC^\sharp$,
\item the functor $([2]\sqcup_{\Lambda^2[2]} [2],\{ 1\to 2\})\to ([2],\{ 1\to 2\})\to \cC^\sharp$,
\item the inclusion $\bI^\flat\to \bI^\sharp\to \cC^\sharp$,
\item the inclusion $([2],\{ 0\to 1, 1\to 2\})\to [2]^\sharp\to \cC^\sharp$, 
\end{rome}
where we use the convention that identities are not explicitly depicted in the sets of marked morphisms.
\end{defn}

\begin{defn}
    The class $\nfib^+$ of \textbf{marked naive fibrations} consists of the morphisms in $\Cat^+_{/\cC^\sharp}$ which have the right lifting property with respect to all morphisms in $\an^+$.
\end{defn}

By unpacking the lifting conditions, we can describe the naive fibrant objects and the naive fibrations between them. The reader may compare to the much simpler case of \cref{fibobj}.

\begin{prop}
A marked functor $P\colon (\cP,E)\to \cC^\sharp$ is naive fibrant if and only if it is a cart-marked Grothendieck fibration.  
\end{prop}

\begin{proof}
We first show that if $P$ satisfies the lifting condition with respect to the generating anodyne extensions in (i)-(iii), then it is a Grothendieck fibration. For this, note that the lifting condition with respect to the map in (i) says that, for every object $a\in \cP$ and every morphism $f\colon c\to Pa$ in $\cC$, there is a marked $P$-lift $g\colon a'\to a$ of $f$. On the other hand, lifting with respect to the maps in (ii) and (iii) translates to the fact that every marked morphism in $(\cP,E)$ is $P$-cartesian. These two facts together imply that $P$ is a Grothendieck fibration. 

Next, we prove that the lifting condition with respect to the maps in (iv) and (v) expresses the fact that $P$ is cart-marked, for which it only remains to show that every $P$-cartesian morphism is in $E$. Let $g'\colon a''\to a$ be a $P$-cartesian morphism. By the above, there is a marked $P$-lift $g\colon a'\to a$ of $Pg'$. As both $g'$ and $g$ are $P$-cartesian lifts of the same morphism,  \cref{propPcart} (iii) gives an isomorphism $h\colon a''\xrightarrow{\cong} a'$ in $\cP$ such that $g'=g h$. Since $P$ lifts against the generating anodyne extension in (iv), we have that the isomorphism $h$~is marked, and then lifting against the functor in (v) ensures that the composite $g'$ of the marked morphisms $g$ and~$h$ is itself marked; hence $g'\in E$ as desired.

The fact that a cart-marked Grothendieck fibration is naive fibrant can be proved by a direct check of the lifting conditions, using that $P$-cartesian morphisms compose by \cref{propPcart} (i), and that isomorphisms are $P$-cartesian by \cref{propPcart} (ii).
\end{proof}

\begin{prop} \label{naivefibvsisofib}
     A marked functor $F\colon (\cP,E_\cP)\to (\cQ,E_\cQ)$ between cart-marked \linebreak Grothendieck fibrations $P\colon (\cP,E_\cP)\to \cC^\sharp$ and $Q\colon (\cQ,E_\cQ)\to \cC^\sharp$ is a naive fibration if and only if it is an isofibration. 
\end{prop}

\begin{proof}
Let $F$ be a naive fibration, and consider an object $a\in \cP$ and an isomorphism $h\colon b\xrightarrow{\cong} Fa$ in $\cQ$. By \cref{propPcart} (ii), we have that $h$ is $Q$-cartesian. As $F$ lifts against $[0]\xrightarrow{1} [1]^\sharp\xrightarrow{Qh} \cC^\sharp$, we get a $P$-cartesian lift $g\colon a'\to a$ of $Qh$ such that $Fg=h$. Moreover, \cref{propPcart} (ii) ensures that $g$ is an isomorphism, and thus $F$ is an isofibration.

Now suppose that $F$ is an isofibration. It is straightforward to check that $F$ lifts against the marked anodyne extensions in (ii)-(v), as $P$ and $Q$ are cart-marked Grothendieck fibrations. It remains to show that there is a lift in the following diagram over $\cC^\sharp$.
\begin{tz}
\node[](1) {$[0]$}; 
\node[right of=1,xshift=.6cm](2) {$(\cP,E_\cP)$}; 
\node[below of=1](1') {$[1]^\sharp$};
\node[below of=2](2') {$(\cQ,E_\cQ)$}; 

\draw[->] (1) to node[above,la]{$a$} (2); 
\draw[->] (1) to (1'); 
\draw[->] (2) to node[right,la]{$F$} (2');
\draw[->] (1') to node[below,la]{$f$} (2');
\draw[->,dashed] (1') to (2);
\end{tz}

Since $P\colon \cP\to \cC$ is a Grothendieck fibration, there is a $P$-cartesian lift $g\colon a'\to a$ of~$Qf$. Then both $Fg$ and $f$ are $Q$-cartesian lifts of $Qf$, as $F$ is a marked functor and as such sends $P$-cartesian morphisms to $Q$-cartesian ones. \cref{propPcart} (iii) then gives an isomorphism $h\colon b\xrightarrow{\cong} Fa'$ in $\cQ$ such that $f=(Fg)h$, which the isofibration $F$ lifts to an isomorphism $g'\colon a''\xrightarrow{\cong} a'$ in $\cP$. Then the composite $gg'\colon a''\to a$ of $P$-cartesian morphisms gives the desired lift.
\end{proof}

We now introduce the class of cofibrations.

\begin{defn}
A marked functor $I\colon (\cA,E_\cA)\to (\cB,E_\cB)$ over $\cC^\sharp$ is a \textbf{marked cofibration} if its underlying functor $I\colon \cA\to \cB$ is injective on objects. A marked functor is a \textbf{marked trivial fibration} if it has the right lifting property with respect to all marked cofibrations.
\end{defn}

In order to apply \cite[Theorem 2.8]{GMSV}, we must ensure that the class of marked cofibrations is cofibrantly generated by a set. We obtain this result upon inspection of the trivial fibrations.

\begin{prop} \label{charmtrivfib}
A marked functor $F\colon (\cP,E_\cP)\to (\cQ,E_\cQ)$ over $\cC^\sharp$ is a marked trivial fibration if and only if the following conditions hold: 
\begin{rome}[leftmargin=1.1cm]
\item its underlying functor $F\colon \cP\to \cQ$ is a trivial fibration in the canonical model structure on~$\Cat$,
\item a morphism $f$ in $\cP$ is marked if and only if $Ff$ in $\cQ$ is marked.
\end{rome} 
\end{prop}

\begin{proof}
If $F$ is a marked trivial fibration, conditions (i) and (ii) can be extracted from its lifting conditions with respect to the marked cofibrations over $\cC^\sharp$
\[ \cA^\flat\to \cB^\flat \quad \text{and} \quad \emptyset\to [1]^\sharp, \]
where $\cA\to \cB$ ranges over the set $\{\emptyset \to [0], \,[0]\sqcup[0]\to [1], \,[1]\sqcup_{[0]\sqcup[0]} [1]\to [1]\}$ of generating cofibrations in the canonical model structure on~$\Cat$.

Conversely, suppose that $F$ satisfies conditions (i) and (ii), and consider a marked cofibration $I\colon (\cA,E_\cA)\to (\cB,E_\cB)$. We show that there is a lift in the following diagram. 
\begin{tz}
\node[](1) {$(\cA,E_\cA)$}; 
\node[right of=1,xshift=1cm](2) {$(\cP,E_\cP)$}; 
\node[below of=1](1') {$(\cB,E_\cB)$};
\node[below of=2](2') {$(\cQ,E_\cQ)$}; 

\draw[->] (1) to (2); 
\draw[->] (1) to node[left,la]{$I$} (1'); 
\draw[->] (2) to node[right,la]{$F$} (2');
\draw[->] (1') to node[below,la]{$G$} (2');
\draw[->,dashed] (1') to (2);
\end{tz}
At the level of underlying categories, such a lift $L\colon \cB\to \cP$ exists as $I$ is a cofibration and $F$ a trivial fibration in the canonical model structure on~$\Cat$ by condition (i). It remains to show that $L$ preserves the marking. For this, let $g$ be a marked morphism in $\cB$; then $Gg$ is marked in $\cQ$. As $FLg=Gg$ is marked, so is $Lg$ by condition (ii), from which we conclude that $L$ is a marked functor. 
\end{proof}

\begin{cor}
    The class of marked cofibrations is cofibrantly generated by a set.
\end{cor}

\cref{charmtrivfib} also yields the following characterization of marked trivial fibrations between fibrant objects.

\begin{cor} \label{prop:trivfibbtwGF}
Let $F\colon (\cP,E_\cP)\to (\cQ,E_\cQ)$ be a marked functor between cart-marked Grothendieck fibrations over $\cC^\sharp$. The following are equivalent:
\begin{rome}[leftmargin=1.1cm]
    \item the marked functor $F$ is a marked trivial fibration, 
    \item the underlying functor $F$ is a trivial fibration in the canonical model structure on~$\Cat$,
    \item the underlying functor $F$ is an equivalence of categories and an isofibration.
\end{rome}
\end{cor}

\begin{proof}
To see that (i) and (ii) are equivalent, it suffices to check that every trivial fibration~$F$ as above satisfies condition (ii) of \cref{charmtrivfib}, and this follows from \cref{trivfibandcartmor}. The fact that (ii) and (iii) are equivalent is immediate from the description of the canonical model structure on~$\Cat$.
\end{proof}

We now study the behavior of trivial fibrations in relation to cofibrations and naive fibrant objects.

\begin{prop} \label{pushofmarkedtrivfib}
The pushout of a marked trivial fibration in $\Cat^+$ along a marked cofibration is a marked trivial fibration. 
\end{prop}

\begin{proof} 
We first want to verify condition (i) in \cref{charmtrivfib}. For this, we show that the pushout of a trivial fibration along a cofibration in the canonical model structure on~$\Cat$ is a trivial fibration. Indeed, trivial fibrations are the equivalences of categories which are surjective on objects; then, as the canonical model structure is left proper (as all objects in $\Cat$ are cofibrant), equivalences are stable under pushout along cofibrations, and surjective-on-objects functors are always stable under pushout. 

It then remains to show that such a pushout satisfies condition (ii) in \cref{charmtrivfib}, which follows from the fact that the set of marked morphisms of the pushout is the pushout of the sets of marked morphisms.
\end{proof}

\begin{prop} \label{trivfibfromGfib}
Let $F\colon (\cP,E_\cP)\to (\cQ,E_\cQ)$ be a marked trivial fibration over $\cC^\sharp$, where $P\colon (\cP,E_\cP)\to \cC^\sharp$ is a cart-marked Grothendieck fibration. Then $Q\colon (\cQ,E_\cQ)\to \cC^\sharp$ is also a cart-marked Grothendieck fibration. 
\end{prop}

\begin{proof}
We first show that $Q\colon \cQ\to \cC$ is a Grothendieck fibration. Let $b\in \cB$ and $f\colon c\to Qb$ be a morphism in $\cC$; we wish to find a $Q$-cartesian lift of $f$. Since $F$ is surjective on objects, there is an object $a\in \cP$ such that $Fa=b$. Using the fact that $P\colon \cP\to \cC$ is a Grothendieck fibration, we obtain a $P$-cartesian lift $g\colon a'\to a$ of $f\colon c\to Qb=PFa$. Then $Fg\colon Fa'\to Fa=b$ is a $Q$-lift of $f$, which is moreover $Q$-cartesian by \cref{trivfibandcartmor}.  

To see that $Q\colon (\cQ,E_\cQ)\to \cC^\sharp$ is cart-marked, we first note that for any morphism $h$ in~$\cQ$ there is an $F$-lift $g$ of $h$, as $F$ is surjective on objects and fully faithful on morphisms. By \cref{charmtrivfib} (ii), the morphism $h=Fg$ is in $E_\cQ$ if and only if $g$ is in $E_\cP$, which is equivalent to $g$ being $P$-cartesian as $P$ is cart-marked. In turn, this happens if and only if $h=Fg$ is $Q$-cartesian by \cref{trivfibandcartmor}.
\end{proof}

We let the class of weak equivalences between cart-marked Grothendieck fibrations be those marked functors whose underlying functors are equivalences of categories. Just as in \cref{defn:we}, this allows us to define the class of weak equivalences between arbitrary objects. 

\begin{defn}
    A marked functor $F\colon(\cP,E_\cP)\to (\cQ,E_\cQ)$ over $\cC^\sharp$ is a \textbf{marked weak equivalence} if there exists a marked naive fibrant replacement of $F$ whose underlying functor is an equivalence of categories. 
\end{defn}

We can now relate the notions of marked trivial fibrations and weak equivalences.

\begin{prop}\label{mtrivfibarewe}
Every marked trivial fibration is a weak equivalence. 
\end{prop}

\begin{proof}
Let $F\colon (\cP,E_\cP)\to (\cQ,E_\cQ)$ be a marked trivial fibration in $\Cat^+_{/\cC^\sharp}$; we wish to find a naive fibrant
replacement of $F$  whose underlying functor is an equivalence of categories. Let $J\colon (\cP,E_\cP)\to (\widehat{\cP},E_{\widehat{\cP}})$ be a naive fibrant replacement of $(\cP,E_\cP)\to \cC^\sharp$ in $\Cat^+_{/\cC^\sharp}$, and consider the following pushout square of marked functors over $\cC^\sharp$. 
\begin{tz}
\node[](1) {$(\cP,E_\cP)$}; 
\node[right of=1,xshift=1cm](2) {$(\cQ,E_\cQ)$}; 
\node[below of=1](1') {$(\widehat{\cP},E_{\widehat{\cP}})$};
\node[below of=2](2') {$(\cQ',E_{\cQ'})$}; 

\draw[->] (1) to node[above,la]{$F$} (2); 
\draw[->] (1) to node[left,la]{$J$} (1'); 
\draw[->] (2) to node[right,la]{$J'$} (2');
\draw[->] (1') to node[below,la]{$F'$} (2');
\node at ($(2')-(.3cm,-.3cm)$) {$\ulcorner$};
\end{tz}

By \cref{pushofmarkedtrivfib}, the functor $F'$ is a marked trivial fibration and, by \cref{trivfibfromGfib}, the induced marked functor $(\cQ',E_{\cQ'})\to \cC^\sharp$ is a cart-marked Grothendieck fibration. Moreover, the marked functor $J'$ is a marked anodyne extension as a pushout of $J$. Hence $(\cQ',E_{\cQ'})$ is a naive fibrant replacement of $(\cQ,E_\cQ)$ and $F'$ a naive fibrant replacement of $F$. By \cref{prop:trivfibbtwGF}, we see that $F'$ is an equivalence of categories and hence conclude that $F$ is a weak equivalence. 
\end{proof}

Before moving on to the main theorem, we explain how one can use the structure we have outlined so far to construct path objects for cart-marked Grothendieck fibrations, which are needed to apply \cite[Theorem 2.8]{GMSV}.

\begin{constr}
Let $P\colon (\cP,E)\to \cC^{\sharp}$ be a cart-marked Grothendieck fibration. We construct a diagram over $\cC^\sharp$
\[ (\cP,E)\to (\mathrm{Path}_\cC (\cP),E')\to (\cP\times_\cC\cP, E\times_{\mathrm{mor}\cC} E).\] 

First, the path object $(\mathrm{Path}_\cC (\cP),E')$ is the marked category whose
\begin{numbered}[leftmargin=1.1cm]
    \item objects are the isomorphisms $g\colon a'\xrightarrow{\cong} a$ in $\cP$ such that $Pg=\id_{Pa}$, 
    \item morphisms $(g\colon a'\xrightarrow{\cong} a)\to (h\colon b'\xrightarrow{\cong} b)$ are morphisms $f'\colon a'\to b'$ and $f\colon a\to b$ making the following diagram commute 
    \begin{tz}
\node[](1) {$a'$}; 
\node[right of=1](2) {$b'$}; 
\node[below of=1](1') {$a$}; 
\node[below of=2](2') {$b$}; 

\draw[->] (1) to node[above,la]{$f'$} (2); 
\draw[->] (1) to node[left,la]{$g$} node[right,la]{$\cong$} (1'); 
\draw[->] (2) to node[right,la]{$h$} node[left,la]{$\cong$} (2'); 
\draw[->] (1') to node[below,la]{$f$} (2');
\end{tz}
    and such that $Pf'=Pf$, 
    \item[(m)] marking consists of pairs $(f',f)$ where both $f'$ and $f$ are marked in $\cP$. 
\end{numbered}

Next, the marked functor $(\cP,E)\to (\mathrm{Path}_\cC (\cP),E')$ sends
\begin{numbered}[leftmargin=1.1cm]
    \item an object $a\in \cP$ to the identity $\id_a$, 
    \item a morphism $f\colon a\to b$ in $\cP$ to the morphism $(f,f)\colon \id_a\to \id_b$. 
\end{numbered}

Finally, the marked functor $(\mathrm{Path}_\cC (\cP),E')\to (\cP\times_\cC\cP, E\times_{\mathrm{mor}\cC} E)$ sends
\begin{numbered}[leftmargin=1.1cm]
    \item an object $a'\xrightarrow{\cong} a$ in $\mathrm{Path}_\cC (\cP)$ to the pair $(a',a)$, 
    \item a morphism $(f',f)$ in $\mathrm{Path}_\cC (\cP)$ to the pair $(f',f)$. 
\end{numbered}
\end{constr}

\begin{prop}\label{mpath}
    Let $P\colon (\cP,E)\to \cC^{\sharp}$ be a cart-marked Grothendieck fibration. Then 
    \[ (\cP,E)\to (\mathrm{Path}_\cC (\cP),E')\to (\cP\times_\cC\cP, E\times_{\mathrm{mor}\cC} E)\]
    is a path object over $\cC^\sharp$.
\end{prop}

\begin{proof}
    By direct inspection, we see that the above composite is the diagonal morphism at $(\cP,E)$ over $\cC^\sharp$. Moreover, one can check that in the factorization
    \[ \cP\to \mathrm{Path}_\cC (\cP)\to \cP\times_\cC\cP, \]
    the first functor is an equivalence of categories and the second one an isofibration. To complete the proof, it suffices to show that marked functor $P'\colon (\mathrm{Path}_\cC (\cP),E')\to \cC^\sharp$ is a cart-marked Grothendieck fibration, since then \cref{naivefibvsisofib} ensures that the second functor above is a naive fibration as desired.

    To do this, first note that a morphism $(f',f)$ in $\mathrm{Path}_\cC (\cP)$ is $P'$-cartesian if and only if both $f'$ and $f$ are $P$-cartesian; hence $P'$ is cart-marked since $P$ is so. We now show that $P'$ is a Grothendieck fibration. Given an object $g\colon a'\xrightarrow{\cong} a$ in $\mathrm{Path}_\cC (\cP)$ and a morphism $k\colon c\to Pa=Pa'$ in $\cC$,  there exist $P$-cartesian lifts $f\colon b\to a$ and $f'\colon b'\to a'$ of $k$. Then, using \cref{propPcart}, we see that $f\colon b\to a$ and $gf'\colon b'\to a$ are two $P$-cartesian lifts of $k$ and so there is an isomorphism $h\colon b'\xrightarrow{\cong} b$ such that $fh=gf'$. This gives a $P'$-cartesian lift $(f',f)\colon h\to g$ of $k$, which concludes the proof. 
\end{proof}

We can now prove the existence of the proposed model structure on $\Cat^+_{/\cC^\sharp}$ for Grothen\-dieck 
fibrations.

\begin{theorem}\label{thm:cartesianMS}
There is a combinatorial model structure on the category $\Cat^+_{/\cC^\sharp}$, denoted by $(\Cat^+_{/\cC^\sharp})_{\mathrm{Grfib}}$, in which 
\begin{rome}[leftmargin=1.1cm]
\item the cofibrations are the injective-on-objects functors,
\item the fibrant objects are the cart-marked Grothendieck fibrations, 
\item the weak equivalences between fibrant objects are the equivalences of categories, 
\item the fibrations between fibrant objects are the isofibrations.
\end{rome}
\end{theorem} 

\begin{proof}
    We apply \cite[Theorem 2.8]{GMSV} using \cite[Proposition 2.21]{GMSV}. First note that the class of marked cofibrations is indeed cofibrantly generated by a set as mentioned in the proof of \cref{charmtrivfib}. Condition (1) (trivial fibrations are weak equivalences) is \cref{mtrivfibarewe}. Condition (2) (2-out-of-6 for $\cW_f$) is clear as the weak equivalences between fibrant objects are equivalences of categories. Similarly, condition (3) (accessibility of $\cW_f$) follows from the fact that the class of equivalences of categories is accessible. Condition (P) (path objects for fibrant objects) is the content of \cref{mpath}. Finally, condition (5) (naive fibrations that are weak equivalences between fibrant objects are trivial fibrations) is ensured by \cref{naivefibvsisofib,prop:trivfibbtwGF}.

This proves that there exists a model structure on $\Cat^+_{/\cC^\sharp}$ with cofibrations, weak equivalences between fibrant objects, and fibrant objects as described in the statement. The description in item (iv) is then given by \cref{naivefibvsisofib}.
\end{proof}

\section{The marked Grothendieck construction}\label{section:markedgroth}

We now consider the Grothendieck construction of a functor $F\colon \cC^\op\to\Cat$. This construction gives a functor 
\[ \textstyle \int_\cC\colon [\cC^{\op},\Cat]\to \Cat_{/\cC}. \]
Since our setting uses marked categories, we need to adapt this construction accordingly, and so we begin by presenting a marked version of the Grothendieck construction
\[ \textstyle \int^+_\cC\colon [\cC^{\op},\Cat]\to \Cat^+_{/\cC^\sharp}\]
which simply endows each category $\int_\cC F$ with the cart-marking; that is, the one where only the cartesian morphisms are marked. The aim of this section is then to prove that this marked Grothendieck construction is a Quillen equivalence between the projective model structure on $[\cC^{\op},\Cat]$ and the model structure $(\Cat^+_{/\cC^\sharp})_{\mathrm{Grfib}}$ for Grothendieck fibrations constructed in \cref{thm:cartesianMS}.

\begin{defn}
We define a functor $\int^+_\cC\colon [\cC^{\op},\Cat]\to \Cat^+_{/\cC^\sharp}$ as follows.
\begin{numbered}[leftmargin=1.1cm]
\item Given a functor $F\colon \cC^{\op}\to \Cat$, its \textbf{marked Grothendieck construction} $\int^+_\cC F$ is the marked category whose 
\begin{numbered}[start=0]
\item objects are pairs $(c,x)$ of an object $c\in \cC$ and an object $x\in Fc$, 
\item morphisms $(c,x)\to (d,y)$ are pairs $(f,\psi)$ of a morphism $f\colon c\to d$ in $\cC$ and a morphism $\psi\colon x\to Ff(y)$ in $Fc$, 
\item[(c)] composition of morphisms $(f,\psi)\colon (c,x)\to (d,y)$ and $(g,\varphi)\colon (d,y)\to (e,z)$ is given by $(gf, Ff(\varphi) \psi)\colon (c,x)\to (e,z)$, where $Ff(\varphi) \psi$ is the composite
\[ x\xrightarrow{\psi} Ff(y)\xrightarrow{Ff(\varphi)} Ff(Fg(z))=F(gf)(z). \]
\item[(m)] a morphism $(f,\psi)$ is marked if $\psi$ is an isomorphism.
\end{numbered}
This comes with a marked canonical projection $\int^+_\cC F\to \cC^\sharp$ onto the first coordinate. 

\item Given a natural transformation $\alpha\colon F\Rightarrow G$ between functors $F,G\colon \cC^{\op}\to \Cat$, the induced marked functor $\int^+_\cC\alpha\colon \int^+_\cC F\to \int^+_\cC G$ over $\cC^\sharp$ sends
\begin{numbered}[start=0]
\item an object $(c,x)$ to the object $(c,\alpha_c(x))$, 
\item a morphism $(f,\psi)\colon (c,x)\to (d,y)$ to the morphism \[ (f,\alpha_c(\psi))\colon (c,\alpha_c(x))\to (d,\alpha_d(y)), \] where $\alpha_c(\psi)\colon \alpha_c(x)\to \alpha_c(Ff(y))=Gf(\alpha_d(y))$.
\end{numbered}
\end{numbered}
\end{defn}

\begin{rem} \label{mGCisfibrant}
    Let $F\colon \cC^{\op}\to \Cat$ be a functor, and consider the projection $P\colon \int_\cC^+ F\to \cC^\sharp$. One can check that a morphism $(f,\psi)$ in $\int_\cC^+ F$ is $P$-cartesian if and only if $\psi$ is invertible. Moreover, given an object $(d,y)$ of $\int_\cC^+ F$ and a morphism $f\colon c\to d$ in $\cC$, a $P$-cartesian lift of $f$ with target $(d,y)$ is given by $(f,\id_{Ff(y)})\colon (c,Ff(y))\to (d,y)$. This shows that the marked Grothendieck construction $P\colon \int_\cC^+ F\to \cC^\sharp$ is a cart-marked Grothendieck fibration. 
\end{rem}

The marked Grothendieck construction functor admits a left adjoint which we introduce shortly, after a brief discussion of marked comma categories and of an adjunction between marked and unmarked categories related to the \emph{natural marking}. 

\begin{rem}
The comma categories of \cref{commaoverP} also admit a marked version. If $P\colon (\cP,E)\to \cC^{\sharp}$ is a marked functor and $c$ is an object in $\cC$, the comma category $c\downarrow P$ can be endowed with a marking $E_c$, where a morphism $g\colon (a',f')\to (a,f)$ in $c\downarrow P$ is marked if and only if the morphism $g\colon a'\to a$ in $\cP$ is marked. 

Note that the functor $d\downarrow P\to c\downarrow P$ induced by pre-composition along a morphism $c\to d$ in $\cC$ preserves this marking. Moreover, if $F\colon (\cP,E_\cP)\to (\cQ,E_\cQ)$ is a marked functor over $\cC^\sharp$, then for every object $c\in \cC$ the induced functor $c\downarrow P\to c\downarrow Q$ also preserves the marking. 
\end{rem}

\begin{rem} \label{rem:isolocadj}
There is a functor $(-)^\natural\colon \Cat\to \Cat^+$ that endows a category $\cC$ with the natural marking; namely, it marks the isomorphisms. This functor has a left adjoint $\Loc\colon \Cat^+\to \Cat$ which sends a marked category $(\cC,E)$ to the localization $\cC[E^{-1}]$ of $\cC$ at the marked morphisms. Recall that such a localization $\cC[E^{-1}]$ comes with a functor $\gamma\colon \cC\to \cC[E^{-1}]$ satisfying the following universal property: for every functor $F\colon \cC\to \cD$ sending morphisms in $E$ to isomorphisms in $\cD$, there is a unique functor $\overline{F}\colon \cC[E^{-1}]\to \cD$ such that $F=\overline{F}\gamma$.
\end{rem}

\begin{defn}
We define a functor $\cT^+_\cC\colon \Cat^+_{/\cC^\sharp}\to [\cC^{\op},\Cat]$ as follows.
\begin{numbered}[leftmargin=1.1cm]
\item Given a marked functor $P\colon (\cP,E)\to \cC^\sharp$, the functor $\cT^+_\cC P\colon \cC^{\op}\to \Cat$  sends
\begin{numbered}[start=0]
\item an object $c\in \cC$ to the category $(c\downarrow P)[E_c^{-1}]$, 
\item a morphism $c\to d$ in $\cC$ to the induced functor $(d\downarrow P)[E_d^{-1}]\to (c\downarrow P)[E_c^{-1}]$. 
\end{numbered}
\item Given a marked functor $F\colon (\cP,E_\cP)\to (\cQ,E_\cQ)$ over $\cC^\sharp$, the  natural transformation $\cT^+_\cC F\colon \cT^+_\cC P\Rightarrow \cT^+_\cC Q$ is the one whose component at an object $c\in \cC$ is given by the induced functor $(c\downarrow P)[(E_\cP)_c^{-1}]\to (c\downarrow Q)[(E_\cQ)_c^{-1}]$.
\end{numbered}
\end{defn}

\begin{theorem}\label{adjunctionmarkedgroth}
There is an adjunction
\begin{tz}
\node[](1) {$[\cC^{\op},\Cat]$}; 
\node[right of=1,xshift=1.6cm](2) {$\Cat^+_{/\cC^{\sharp}}$}; 

\draw[->] ($(1.east)-(0,5pt)$) to node[below,la]{$\int^+_\cC$} ($(2.west)-(0,5pt)$);
\draw[->] ($(2.west)+(0,5pt)$) to node[above,la]{$\cT^+_\cC$} ($(1.east)+(0,5pt)$);

\node[la] at ($(1.east)!0.5!(2.west)$) {$\bot$};
\end{tz}
\end{theorem}

\begin{proof} 
In order to construct the unit $\eta\colon \id_{\Cat^+_{/\cC^\sharp}}\Rightarrow \int_\cC^+\cT_\cC^+$, we begin by computing the marked Grothendieck construction $\int^+_\cC \cT^+_\cC(P)$ of a given marked functor $P\colon (\cP,E)\to \cC^\sharp$. This is the marked category whose underlying category is the localization of the category whose 
\begin{numbered}[leftmargin=1.1cm]
\item objects are tuples $(c,a,g)$ of an object $c\in \cC$, an object $a\in \cP$, and a morphism $g\colon c\to Pa$ in $\cC$, 
\item morphisms $(c',a',g')\to (c,a,g)$ are pairs $(f,h)$ of a morphism $f\colon c'\to c$ in $\cC$ and a morphism $h\colon a'\to a$ in $\cP$ such that $gf=(Ph)g'$,
\end{numbered}
localized at the morphisms $(f,h)$ with $h\in E$, and whose 
\begin{numbered}[leftmargin=1.1cm]
\item[(m)] a morphism $(f,h)$ is marked if either $h$ is in $E$ or $h$ is an isomorphism in $\cP$; additionally, the formal inverses $(f,h)^{-1}$ with $h\in E$ are marked.
\end{numbered}
Then the component of the unit at the marked functor $P$ is given by the marked functor $\eta_P\colon (\cP,E)\to \int^+_\cC \cT^+_\cC(P)$ over $\cC^\sharp$ sending
\begin{numbered}[leftmargin=1.1cm]
\item an object $a\in \cP$ to the tuple $(Pa,a,\id_{Pa})$, 
\item a morphism $h\colon a'\to a$ in $\cP$ to the pair $(Ph,h)$. 
\end{numbered}
Note that if $h$ is marked, then so is $(Ph,h)$ and so $\eta_P$ preserves the marking. By unpacking the definitions, one can verify that these $\eta_P$ assemble into a natural transformation $\eta\colon \id_{\Cat^+_{/\cC^\sharp}}\Rightarrow \int_\cC^+\cT_\cC^+$.

We now focus on constructing the counit $\varepsilon\colon \cT_\cC^+\int_\cC^+\Rightarrow\id_{[\cC^{\op},\Cat]}$, for which we first compute the functor $\cT_\cC^+\int_\cC^+ F\colon \cC^{\op}\to \Cat$ for any given functor $F\colon \cC^{\op}\to \Cat$. It sends 
\begin{numbered}[leftmargin=1.1cm]
    \item an object $c\in \cC$ to the category $\cT_\cC^+\int_\cC^+ F(c)$, which is the localization of the category whose
    \begin{numbered}[start=0]
        \item objects are tuples $(d,y,f)$ of an object $d\in \cC$, an object $y\in Fd$, and a morphism $f\colon c\to d$, 
        \item morphisms $(d',y',f')\to (d,y,f)$ are pairs $(g,\psi)$ of a morphism $g\colon d'\to d$ in $\cC$ and a morphism $\psi\colon y'\to Fg(y)$ such that $gf'=f$,
\end{numbered}
    localized at the morphisms $(g,\psi)$ such that $\psi$ is invertible. 
    \item a morphism $h\colon c'\to c$ in $\cC$ to the functor $h^*\colon \cT_\cC^+\int_\cC^+ F(c)\to \cT_\cC^+\int_\cC^+ F(c')$ sending 
     \begin{numbered}[start=0]
        \item an object $(d,y,f)$ to the object $(d,y,fh)$, 
        \item a morphism $(g,\psi)$ to the morphism $(g,\psi)$. 
    \end{numbered}
\end{numbered}
Then the component of the counit at the functor $F$ is given by the natural transformation $\varepsilon_F\colon \cT_\cC^+\int_\cC^+ F\Rightarrow F$ whose component at $c\in \cC$ is given by the functor $\cT_\cC^+\int_\cC^+ F(c)\to Fc$ sending 
\begin{numbered}[leftmargin=1.1cm]
    \item an object $(d,y,f)$ to the object $Ff(y)\in Fc$, 
    \item a morphism $(g,\psi)\colon (d',y',f')\to (d,y,f)$ to the morphism 
 \[ Ff'(\psi)\colon Ff'(y')\to (Ff')(Fg)(y)=F(gf')(y)=Ff(y). \] 
\end{numbered}
Note that if $(g,\psi)$ is invertible in $\cT_\cC^+\int_\cC^+ F(c)$, i.e., if $\psi$ is invertible, then so is $Ff'(\psi)$; hence this is well-defined. Once again, a careful study of the definitions shows that these $(\varepsilon_F)_c$ assemble into natural transformations $\varepsilon_F\colon \cT_\cC^+\int_\cC^+ F\Rightarrow F$, which themselves assemble into a natural transformation $\varepsilon\colon \cT_\cC^+\int_\cC^+\Rightarrow\id_{[\cC^{\op},\Cat]}$.

A tedious computation further shows that the unit and counit as defined above satisfy the triangle identities.
\end{proof}

Our goal is now to show that the above adjunction is in fact a Quillen equivalence. For this, we first recall the definition of the projective model structure on $[\cC^{\op},\Cat]$. Just as the model structure on $[\cC^{\op},\Set]$ of \cref{projMSset} relied on the trivial model structure on $\Set$, the projective model structure we now state relies on the canonical model structure on $\Cat$ from \cref{thm:MSCat}. 

\begin{theorem}
    There is a model structure on $[\cC^{\op},\Cat]$, called the projective model structure and denoted by $[\cC^{\op},\Cat]_{\mathrm{proj}}$, in which
    \begin{rome}[leftmargin=1.1cm]
\item the weak equivalences are the level-wise equivalences of categories, 
\item the fibrations are the level-wise isofibrations. 
\end{rome}
\end{theorem}

\begin{theorem} \label{thm:markedGCQE}
The adjunction 
\begin{tz}
\node[](1) {$[\cC^{\op},\Cat]_{\mathrm{proj}}$}; 
\node[right of=1,xshift=2.4cm](2) {$(\Cat^+_{/\cC^\sharp})_{\mathrm{Grfib}}$}; 

\draw[->] ($(1.east)-(0,5pt)$) to node[below,la]{$\int^+_\cC$} ($(2.west)-(0,5pt)$);
\draw[->] ($(2.west)+(0,5pt)$) to node[above,la]{$\cT^+_\cC$} ($(1.east)+(0,5pt)$);

\node[la] at ($(1.east)!0.5!(2.west)$) {$\bot$};
\end{tz}
is a Quillen equivalence.
\end{theorem}

\begin{proof}
    To show that $\int^+_\cC$ is right Quillen, we must show that it preserves fibrations and trivial fibrations.  Let $\alpha\colon F\Rightarrow G$ be a (trivial) fibration in $[\cC^{\op},\Cat]_{\mathrm{proj}}$; i.e., for every $c\in \cC$, the functor $\alpha_c\colon Fc\to Gc$ is a (trivial) fibration in $\Cat$. Since the marked functor $\int_\cC^+ \alpha\colon \int_\cC^+ F\to \int_\cC^+ G$ acts as $\alpha_c$ on each component, we get that its underlying functor is a (trivial) fibration in $\Cat$. Furthermore, as explained in \cref{mGCisfibrant}, the functor $\int^+_\cC$ sends every functor in $[\cC^{\op},\Cat]$ to a cart-marked Grothendieck fibration over $\cC^\sharp$. Hence, we can use \cref{naivefibvsisofib,prop:trivfibbtwGF} to conclude that $\int_\cC^+ \alpha$ is a (trivial) fibration in $(\Cat^+_{/\cC^\sharp})_{\mathrm{Grfib}}$. 

Next, we show that the unit is a weak equivalence. To achieve this, we first restrict to fibrant objects, and prove that $\eta_P\colon (\cP,E)\to \int_\cC ^+ \cT_\cC^+(P)$ is an equivalence of categories for any cart-marked Grothendieck fibration $P\colon (\cP,E)\to \cC^\sharp$.

To show that $\eta_P$ is essentially surjective on objects, let $(c,a,g)$ be an object in $\int_\cC ^+ \cT_\cC^+(P)$, where $c\in \cC$ and $a\in \cP$ are objects and $g\colon c\to Pa$ is a morphism in $\cC$. Since $P$ is a Grothendieck fibration, there is a $P$-cartesian lift $h\colon a'\to a$ of $g$. In particular, we have that $Pa'=c$ and, since $P$ is cart-marked, that $h$ is marked. Hence  $(\id_c,h)\colon \eta_P(a')=(Pa',a',\id_{Pa'})\to (c,a,g)$ gives the desired isomorphism in $\int_\cC ^+ \cT_\cC^+(P)$. 

Given objects $a,a'\in \cP$, a morphism $(Pa',a',\id_{Pa'})\to (Pa,a,\id_{Pa})$ in $\int_\cC ^+ \cT_\cC^+(P)$ consists of a pair $(f,h)$ of morphisms $f\colon Pa'\to Pa$ in $\cC$ and $h\colon a'\to a$ in $\cP$ such that $f=Ph$. We then see that it is uniquely determined by the morphism $h$ in $\cP$, which shows that $\eta_P$ is fully faithful on morphisms. This shows that $\eta_P$ is an equivalence of categories. 

Now, given any marked functor $P\colon (\cP,E)\to \cC^\sharp$, let $(\cP,E)\to (\widehat{\cP},\widehat{E})$ be a fibrant replacement of $P$ in $(\Cat^+_{/\cC^\sharp})_{\mathrm{Grfib}}$ with $\widehat{P}\colon (\widehat{\cP},\widehat{E})\to \cC^\sharp$. This gives a commutative square of marked functors over $\cC^\sharp$
\begin{tz}
\node[](1) {$(\cP,E)$}; 
\node[right of=1,xshift=1.1cm](2) {$\int^+_\cC \cT^+_\cC(P)$}; 
\node[below of=1](1') {$(\widehat{\cP},\widehat{E})$};
\node[below of=2](2') {$\int^+_\cC \cT^+_\cC(\widehat{P})$}; 

\draw[->] (1) to (2); 
\draw[->] (1) to node[below,la, sloped]{$\sim$} (1'); 
\draw[->] (2) to node[above,la, sloped,yshift=-2pt]{$\sim$} (2');
\draw[->] (1') to node[below,la]{$\simeq$} (2');
\end{tz}
where the bottom functor $\eta_{\widehat{P}}$ is an equivalence of categories since $\widehat{P}\colon (\widehat{\cP},\widehat{E})\to \cC^\sharp$ is a cart-marked Grothendieck fibration, and the left-hand functor is a weak equivalence by construction. Moreover, the right-hand functor is also a weak equivalence, as the fact that all objects in $(\Cat^+_{/\cC^\sharp})_{\mathrm{Grfib}}$ are cofibrant and all objects in $[\cC^{\op},\Cat]_{\mathrm{proj}}$ are fibrant ensures that both $\cT^+_\cC$ and $\int^+_\cC$ preserve weak equivalences by Ken Brown's lemma \cite[Lemma 1.1.12]{hovey}. Hence by $2$-out-of-$3$, we get that $\eta_P$ is also a weak equivalence, as desired. 

To conclude the proof, by the dual version of \cite[Lemma 3.3]{mehmet} it suffices to show that $\int^+_\cC$ reflects weak equivalences. Let $\alpha\colon F\Rightarrow G$ be a natural transformation in $[\cC^{\op},\Cat]$ such that $\int_\cC^+ \alpha$ is an equivalence of categories; we need to prove that $\alpha_c \colon Fc\to Gc$ is an equivalence of categories for every object $c\in\cC$. 

We first show that $\alpha_c$ is essentially surjective on objects. Given an object $y\in Gc$, consider the object $(c,y)$ of $\int_\cC^+ G$. Since $\int_\cC^+ \alpha$ is essentially surjective on objects, there is an object $(c',x)$ in $\int_\cC^+ F$ and an isomorphism $(f,\varphi)\colon (c',\alpha_{c'}(x))\xrightarrow{\cong} (c,y)$ in $\int_\cC^+ G$, i.e., an isomorphism $f\colon c'\xrightarrow{\cong} c$ in $\cC$ and an isomorphism $\varphi\colon \alpha_{c'}(x)\xrightarrow{\cong} Gf(y)$ in $Fc'$. Then $Ff^{-1}(x)$ is an object of $Fc$, and we have an isomorphism \[
Gf^{-1}(\varphi)\colon \alpha_c (Ff^{-1}(x)) =Gf^{-1}(\alpha_{c'}(x))\to Gf^{-1} (Gf(y))=y. \]

To see that $\alpha_c$ is fully faithful on morphisms, note that for all objects $x,x'\in Fc$, the bijection 
\[ \textstyle (\int^+_\cC F)((c,x),(c,x'))\cong (\int^+_\cC G)((c,\alpha_c (x)),(c,\alpha_c (x')))\]
induced by $\int_\cC^+\alpha$ restricts to a bijection 
\[ (Fc)(x,x')\cong (Gc)(\alpha_c (x),\alpha_c (x')) \]
induced by $\alpha_c$ when we only allow $\id_c$ in the first component. This concludes the proof.
\end{proof}

\begin{rem}
    When $\cC=[0]$ is the terminal category, there are canonical identifications 
    \[ \Cat\cong [[0],\Cat] \quad \text{and} \quad \Cat^+\cong \Cat^+_{/[0]}. \]
    Then the adjunction $\cT^+_{[0]}\dashv \int^+_{[0]}$ can be identified with the adjunction from \cref{rem:isolocadj}
    \begin{tz}
\node[](1) {$\Cat$}; 
\node[right of=1,xshift=1.1cm](2) {$\Cat^+$}; 

\draw[->] ($(1.east)-(0,5pt)$) to node[below,la]{$(-)^\natural$} ($(2.west)-(0,5pt)$);
\draw[->] ($(2.west)+(0,5pt)$) to node[above,la]{$\Loc$} ($(1.east)+(0,5pt)$);

\node[la] at ($(1.east)!0.5!(2.west)$) {$\bot$};
\end{tz}
\end{rem}

We consider the category $\Cat^+$ to be endowed with the model structure $(\Cat^+_{/[0]})_\mathrm{Grfib}$. As a consequence of the above remark, by taking $\cC=[0]$ in \cref{thm:markedGCQE} we obtain the following result. 

\begin{cor} \label{cor:locisoQP}
The adjunction
    \begin{tz}
\node[](1) {$\Cat$}; 
\node[right of=1,xshift=1.1cm](2) {$\Cat^+$}; 

\draw[->] ($(1.east)-(0,5pt)$) to node[below,la]{$(-)^\natural$} ($(2.west)-(0,5pt)$);
\draw[->] ($(2.west)+(0,5pt)$) to node[above,la]{$\Loc$} ($(1.east)+(0,5pt)$);

\node[la] at ($(1.east)!0.5!(2.west)$) {$\bot$};
\end{tz}
is a Quillen equivalence.
\end{cor}

\section{Comparison with infinity: discrete case}\label{section:discinfty}

We now want to compare the model structure on $\Cat_{/\cC}$ for discrete fibrations constructed in \cref{thm:discreteMS} to its $\infty$-counterpart: the contravariant model structure on  the category $\sSet_{/N\cC}$ of simplicial sets over the nerve $N\cC$, denoted here by $(\sSet_{/N\cC})_\mathrm{contra}$. This model structure, constructed by Lurie in \cite[Proposition 2.1.4.7]{HTT}, has as cofibrations the monomorphisms of simplicial sets, and as fibrant objects the \emph{right fibrations}---the $\infty$-analogue of discrete fibrations. 

The goal of this section is to study the following square of Quillen adjunctions.
\begin{diagram} \label{QP1}
    \node[](1) {$[\cC^{\op},\Set]_{\mathrm{proj}}$}; 
    \node[below of=1,yshift=-.5cm](2) {$[\cC^\op,\sSet]_\mathrm{proj}$}; 
    \node[right of=1,xshift=2.6cm](3) {$(\Cat_{/\cC})_{\mathrm{discfib}}$}; 
    \node[below of=3,yshift=-.5cm](4) {$(\sSet_{/N\cC})_{\mathrm{contra}}$};

    \draw[->] ($(1.east)-(0,5pt)$) to node[below,la]{$\int_\cC$} ($(3.west)-(0,5pt)$);
\draw[->] ($(3.west)+(0,5pt)$) to node[above,la]{$\cT_\cC$} ($(1.east)+(0,5pt)$);
\node[la] at ($(1.east)!0.5!(3.west)$) {$\bot$};
\draw[->] ($(2.east)-(0,5pt)$) to node[below,la]{$\mathrm{Un}_\cC$} ($(4.west)-(0,5pt)$);
\draw[->] ($(4.west)+(0,5pt)$) to node[above,la]{$\mathrm{St}_\cC$} ($(2.east)+(0,5pt)$);
\node[la] at ($(2.east)!0.5!(4.west)$) {$\bot$};

\draw[->] ($(3.south)+(5pt,0)$) to node[right,la]{$N$} ($(4.north)+(5pt,0)$); 
\draw[->] ($(4.north)-(5pt,0)$) to node[left,la]{$c$} ($(3.south)-(5pt,0)$); 
\node[la] at ($(3.south)!0.5!(4.north)$) {\rotatebox{90}{$\bot$}};
\draw[->] ($(1.south)+(5pt,0)$) to node[right,la]{$\iota_*$} ($(2.north)+(5pt,0)$); 
\draw[->] ($(2.north)-(5pt,0)$) to node[left,la]{$(\pi_0)_*$} ($(1.south)-(5pt,0)$); 
\node[la] at ($(1.south)!0.5!(2.north)$) {\rotatebox{90}{$\bot$}};
\end{diagram}
The adjunction at the top, presented in \cref{adjclassicalGC}, has the classical Grothendieck construction as the right adjoint, and is a Quillen equivalence by \cref{thm:discreteQuillenequiv}. The bottom adjunction is referred to as \emph{straightening-unstraightening}, and \cite[Theorem 2.2.1.2]{HTT} shows it is a Quillen equivalence when considering the projective model structure on functors into $\sSet$ endowed with the Kan-Quillen model structure. For the left side, recall that the adjunction
\begin{tz}
\node[](1) {$\Set$}; 
\node[right of=1,xshift=.9cm](2) {$\sSet$}; 

\draw[->] ($(1.east)-(0,5pt)$) to node[below,la]{$\iota$} ($(2.west)-(0,5pt)$);
\draw[->] ($(2.west)+(0,5pt)$) to node[above,la]{$\pi_0$} ($(1.east)+(0,5pt)$);

\node[la] at ($(1.east)!0.5!(2.west)$) {$\bot$};
\end{tz}
is a \emph{Quillen reflection} (i.e., a Quillen pair such that the derived counit is a level-wise weak equivalence), where $\iota$ denotes the canonical inclusion. Hence, it induces by post-composition a Quillen reflection that features as the left-hand adjunction in \eqref{QP1}. Finally, for the right side, note that the adjunction 
\begin{tz}
\node[](1) {$\Cat$}; 
\node[right of=1,xshift=1cm](2) {$\sSet$}; 

\draw[->] ($(1.east)-(0,5pt)$) to node[below,la]{$N$} ($(2.west)-(0,5pt)$);
\draw[->] ($(2.west)+(0,5pt)$) to node[above,la]{$c$} ($(1.east)+(0,5pt)$);

\node[la] at ($(1.east)!0.5!(2.west)$) {$\bot$};
\end{tz}
induces an adjunction between slices over $\cC$ and $N\cC$; this gives the right-hand adjunction in \eqref{QP1}. We show in \cref{discQuillenrefl} that this is a Quillen reflection, and in \cref{prop:leftQuillencommute1,prop:rightQuillencommute1} that the squares of left and right adjoints in \eqref{QP1} commute up to a level-wise weak equivalence.

We begin by studying the adjunction $c\dashv N$.

\begin{prop}\label{discQuillenrefl}
The adjunction 
\begin{tz}
\node[](1) {$(\Cat_{/\cC})_\mathrm{discfib}$}; 
\node[right of=1,xshift=2.6cm](2) {$(\sSet_{/N\cC})_\mathrm{contra}$}; 

\draw[->] ($(1.east)-(0,5pt)$) to node[below,la]{$N$} ($(2.west)-(0,5pt)$);
\draw[->] ($(2.west)+(0,5pt)$) to node[above,la]{$c$} ($(1.east)+(0,5pt)$);

\node[la] at ($(1.east)!0.5!(2.west)$) {$\bot$};
\end{tz}
is a Quillen reflection. 
\end{prop}

\begin{proof}
    First note that $c$ trivially preserves cofibrations, as all functors are cofibrations in $(\Cat_{/\cC})_\mathrm{discfib}$. By \cite[Proposition E.2.14]{JoyalVolumeII}, it remains to show that $N$ preserves fibrations between fibrant objects. However, in both model structures these are determined by the right lifting property with respect to a class of anodyne extensions, so it suffices to prove that $c$ preserves anodyne extensions, which we now do.
    
    The anodyne extensions of the model structure $(\sSet_{/N\cC})_\mathrm{contra}$ are cofibrantly generated by the morphisms $\Lambda^t[n]\to \Delta[n]\to N\cC$, for $n\geq 1$ and $0<t\leq n$ (see \cite[Lemma 4.3 and \S 4.4]{kim} for a justification of this claim, and \cite[Theorem 2.17]{kim} for the fact that these determine the fibrations between fibrant objects). When we apply $c$, we see that most of these generating morphisms are mapped to the identity, except for the following:
    \begin{enumerate}[leftmargin=1.1cm]
        \item the map for $n=1$ and $t=1$ is sent to the anodyne extension $[0]\xrightarrow{1} [1]\to \cC$, 
        \item the map for $n=2$ and $t=2$ is sent to a composite of pushouts of the anodyne extensions $[0]\xrightarrow{1} [1]\to \cC$ and $[1]\sqcup_{[0]}[1]\to [1]\to \cC$,
        \item the map for $n=3$ and $t=3$ is sent to a pushout of the anodyne extension $[1]\sqcup_{[0]}[1]\to [1]\to \cC$.
    \end{enumerate}
    In particular, all of these maps are anodyne extensions in $(\Cat_{/\cC})_\mathrm{discfib}$, and so $c$ is left Quillen.

    Moreover, since $cN\cong \id_{\Cat}$ we have that the component of the (derived) counit at a discrete fibration $P\colon \cP\to \cC$ is an isomorphism, showing that $c\dashv N$ is a Quillen reflection.
\end{proof}

The commutativity of the square \eqref{QP1} is quite hard to unpack directly, due to the complexity of the straightening-unstraightening adjunction $\mathrm{St}_\cC\dashv \mathrm{Un}_\cC$. Hence, we first work with a simplified construction 
\begin{tz}
\node[](1) {$[\cC^{\op},\sSet]_\mathrm{proj}$}; 
\node[right of=1,xshift=2.6cm](2) {$(\sSet_{/N\cC})_\mathrm{contra}$}; 

\draw[->] ($(1.east)-(0,5pt)$) to node[below,la]{$N_{(-)}(\cC)$} ($(2.west)-(0,5pt)$);
\draw[->] ($(2.west)+(0,5pt)$) to node[above,la]{$\cF_{(-)}(\cC)$} ($(1.east)+(0,5pt)$);

\node[la] at ($(1.east)!0.5!(2.west)$) {$\bot$};
\end{tz}
given by the \emph{relative nerve} and its left adjoint introduced in \cite[Definition 3.2.5.2 and Remark 3.2.5.5]{HTT}. By \cite[Proposition 3.2.5.18(1)]{HTT}, this provides another Quillen equivalence between the projective and contravariant model structures. Moreover, since the left adjoint $\cF_{(-)}(\cC)$ is related by a level-wise weak equivalence to the straightening functor $\mathrm{St}_{\cC}$ by \cite[Lemma 3.2.5.17(1)]{HTT}, it is a suitable replacement to study the commutativity of the square \eqref{QP1} of adjoint functors.

\begin{prop} \label{prop:leftadjcommute1}
The following square of left adjoints commutes up to isomorphism.
    \begin{tz}
\node[](1) {$\sSet_{/N\cC}$}; 
\node[right of=1,xshift=1.8cm](2) {$[\cC^{\op},\sSet]$}; 
\node[below of=1](1') {$\Cat_{/\cC}$};
\node[below of=2](2') {$[\cC^{\op},\Set]$}; 

\draw[->] (1) to node[above,la]{$\cF_{(-)}(\cC)$} (2); 
\draw[->] (1) to node[left,la]{$c$} (1'); 
\draw[->] (2) to node[right,la]{$(\pi_0)_*$} (2');
\draw[->] (1') to node[below,la]{$\cT_\cC$} (2');
\end{tz}
\end{prop}

\begin{proof}
    Let us denote by $\int_{\Delta} N\cC$ the category of elements of the simplicial set $N\cC$. Note that $\sSet_{/N\cC}$ is a presheaf category, as there is an isomorphism $\sSet_{/N\cC}\cong \Set^{(\int_{\Delta} N\cC)^{\op}}$. Hence, as every functor in the above square is a left adjoint, it is enough to show that their values agree on all representables $\Delta[n]\to N\cC$.

    A representable $\sigma\colon \Delta[n]\to N\cC$ is sent by $c$ to the corresponding functor $\sigma\colon [n]\to \cC$, which is in turn sent by $\cT_\cC$ to the functor $\pi_0((-)\downarrow \sigma)\colon \cC^{\op}\to \Set$. On the other hand, we can express the action of $(\pi_0)_*\cF_{(-)}(\cC)$ on $\sigma$ as the composite of the left vertical map and the two bottom horizontal maps in the diagram below.
\begin{tz}
\node[](1) {$\sSet_{/\Delta[n]}$}; 
\node[right of=1,xshift=2cm](2) {$[[n]^{\op},\sSet]$}; 
\node[right of=2,xshift=2cm](3) {$[[n]^{\op},\Cat]$}; 
\node[below of=1](1') {$\sSet_{/N\cC}$};
\node[below of=2](2') {$[\cC^{\op},\sSet]$}; 
\node[below of=3](3') {$[\cC^{\op},\Set]$};

\draw[->] (1) to node[above,la]{$\cF_{(-)}([n])$} (2); 
\draw[->] (2) to node[above,la]{$(\pi_0)_*$} (3); 
\draw[->] (1) to node[left,la]{$\sigma_!$} (1'); 
\draw[->] (2) to node[right,la]{$\sigma_!$} (2');
\draw[->] (3) to node[right,la]{$\sigma_!$} (3');
\draw[->] (1') to node[below,la]{$\cF_{(-)}(\cC)$} (2');
\draw[->] (2') to node[below,la]{$(\pi_0)_*$} (3');
\end{tz}
Now, the left-hand square commutes up to isomorphism by \cite[Remark 3.2.5.8]{HTT}, and the right-hand square commutes up to isomorphism as it is straightforward to check that the corresponding square of right adjoints commutes. Hence, this allows for the following rewriting  \[ (\pi_0)_*\cF_{\sigma}(\cC)\cong \sigma_!(\pi_0)_*\cF_{\id_{\Delta[n]}}([n]). \]

By \cite[Example 3.2.5.6]{HTT}, there is an isomorphism $\cF_{\id_{\Delta[n]}}([n])\cong N((-)\downarrow [n])$ and so we have
\[ (\pi_0)_*\cF_{\id_{\Delta[n]}}([n])\cong \pi_0 N((-)\downarrow [n])\cong \pi_0((-)\downarrow [n]). \]  Moreover, since post-composing with the connected component functor commutes with left Kan extensions, we get that 
\[ \sigma_! (\pi_0((-)\downarrow [n]))\cong \pi_0(\sigma_!((-)\downarrow [n])). \] To obtain the desired claim, it then suffices to compute the left Kan extension along $\sigma\colon [n]\to \cC$ and show that 
\[ \sigma_! ((-)\downarrow [n])\cong (-)\downarrow \sigma \]
which is the content of the independent technical \cref{lem:LKE} below.
\end{proof}

\begin{lem} \label{lem:LKE}
   The left Kan extension of  $[-]\downarrow [n]\colon [n]^{\op}\to \Cat$ along $\sigma\colon [n]\to\cC$ is given by $(-)\downarrow \sigma\colon \cC^{\op}\to \Cat$.
\end{lem}

\begin{proof}
We start by using the formula for pointwise left Kan extensions (see for instance \cite[Theorem 6.2.1]{emilycontext}): given $c\in \cC$, we can rewrite
\begin{align*}
    \sigma_!((-)\downarrow [n]))(c) &\cong \mathrm{colim}(c\downarrow \sigma\to [n]^{\op}\xrightarrow{(-)\downarrow [n]} \Cat) \\
    &\cong \textstyle\mathrm{coeq}\left(\coprod_{0\leq i\leq j\leq n} \cC(c,\sigma(i))\times (j\downarrow [n])\rightrightarrows \coprod_{0\leq i\leq n} \cC(c,\sigma(i))\times (i\downarrow [n])\right).
\end{align*} 
So there is a canonical functor 
\[ \sigma_!((-)\downarrow [n])(c)\to c\downarrow \sigma \]
induced, for each $0\leq i\leq n$, by the functor $\cC(c,\sigma(i))\times (i\downarrow [n])\to c\downarrow \sigma$ that sends 
\begin{numbered}[leftmargin=1.1cm]
    \item an object $(c\xrightarrow{f}\sigma(i),i\xrightarrow{g} j)$ in $\cC(c,\sigma(i))\times (i\downarrow [n])$ to the composite morphism $c\xrightarrow{f} \sigma(i)\xrightarrow{\sigma(g)} \sigma(j)$ in $c\downarrow \sigma$, 
    \item a morphism $(f,h)\colon (f,g)\to (f,hg)$ in $\cC(c,\sigma(i))\times (i\downarrow [n])$ with $h\colon j\to k$ in $[n]$ to the morphism $h\colon \sigma(g)\circ f\to \sigma(hg)\circ f$ in $c\downarrow \sigma$.
\end{numbered}
These functors coequalize the two parallel functors in the coequalizer diagram above and hence induce a functor from $\sigma_!((-)\downarrow [n])(c)$ as desired. Indeed, given an object $(c\xrightarrow{f} \sigma(i), j\xrightarrow{h} k)$ in $\cC(c,\sigma(i))\times (j\downarrow [n])$ and the unique map $i\xrightarrow{g} j$ in $[n]$, the two parallel functors send the tuple $(f,h)$ to 
\[ (c\xrightarrow{f} \sigma(i)\xrightarrow{\sigma(g)} \sigma(j), j\xrightarrow{h} k) \quad \text{and} \quad (c\xrightarrow{f} \sigma(i), i\xrightarrow{g}j\xrightarrow{h} k), \]
both of which are mapped to the same object of $c\downarrow \sigma$.

We now construct a functor $c\downarrow\sigma \to \sigma_!((-)\downarrow [n])(c)$ in the other direction which sends 
\begin{numbered}[leftmargin=1.1cm]
    \item an object $c\xrightarrow{f}\sigma(i)$ in $c\downarrow \sigma$ to the object of $\sigma_!((-)\downarrow [n])(c)$ represented by $(c\xrightarrow{f}\sigma(i), \id_i)$, 
    \item a morphism $g\colon f\to \sigma(g)\circ f$ in $c\downarrow \sigma$ to the morphism of $\sigma_!((-)\downarrow [n])(c)$ represented by $(f,g)\colon (f,\id_i)\to (f,g)$, 
\end{numbered}
where we use that $(f,g)\sim (\sigma(g)\circ f, \id_j)$ as objects of the coequalizer. Using this identification, both functors compose to the identity and so we get an isomorphism
\[ \sigma_!((-)\downarrow [n])(c)\cong c\downarrow \sigma \]
for each $c\in \cC$, yielding the desired isomorphism
\[ \sigma_!((-)\downarrow [n])\cong (-)\downarrow \sigma. \qedhere \]
\end{proof}

We can now use \cref{prop:leftadjcommute1} to obtain the commutativity of the square of left adjoints of interest.

\begin{prop} \label{prop:leftQuillencommute1}
The following square of left Quillen functors commutes up to a level-wise weak equivalence.
    \begin{tz}
\node[](1) {$(\sSet_{/N\cC})_\mathrm{contra}$}; 
\node[right of=1,xshift=2.6cm](2) {$[\cC^{\op},\sSet]_\mathrm{proj}$}; 
\node[below of=1](1') {$(\Cat_{/\cC})_\mathrm{discfib}$};
\node[below of=2](2') {$[\cC^{\op},\Set]_\mathrm{proj}$}; 

\draw[->] (1) to node[above,la]{$\mathrm{St}_\cC$} node[below,la]{$\simeq_{QE}$} (2); 
\draw[->] (1) to node[left,la]{$c$} (1'); 
\draw[->] (2) to node[right,la]{$(\pi_0)_*$} (2');
\draw[->] (1') to node[below,la]{$\cT_\cC$} node[above,la]{$\simeq_{QE}$} (2');
\end{tz}
\end{prop}

\begin{proof}
    By \cite[Lemma 3.2.5.17(1)]{HTT}, there is a natural transformation 
    \[ \mathrm{St}_\cC \Rightarrow \cF_{(-)}(\cC) \]
    whose components are weak equivalences in $[\cC^{\op},\sSet]_\mathrm{proj}$. Now, note that the left Quillen functor 
    \[ \pi_0\colon \sSet\to \Set\]
    preserves all weak equivalences since all objects are cofibrant in $\sSet$. Hence the functor 
    \[ (\pi_0)_*\colon [\cC^{\op},\sSet]_\mathrm{proj}\to [\cC^{\op},\Set]_\mathrm{proj}\]
    also preserves all weak equivalences as they are defined level-wise. This implies that the natural transformation obtained by whiskering
    \[ (\pi_0)_*\mathrm{St}_\cC \Rightarrow (\pi_0)_*\cF_{(-)}(\cC) \]
     is level-wise a weak equivalence in $[\cC^{\op},\Set]_\mathrm{proj}$. Combining this with the isomorphism from \cref{prop:leftadjcommute1}, we get the desired result. 
\end{proof}

\begin{theorem}\label{prop:rightQuillencommute1}
    The following square of homotopically fully faithful right Quillen functors commutes up to a level-wise weak equivalence. 
    \begin{tz}
\node[](1) {$[\cC^{\op},\Set]_{\mathrm{proj}}$}; 
\node[right of=1,xshift=2.6cm](2) {$(\Cat_{/\cC})_{\mathrm{discfib}}$}; 
\node[below of=1](1') {$[\cC^\op,\sSet]_\mathrm{proj}$};
\node[below of=2](2') {$(\sSet_{/N\cC})_{\mathrm{contra}}$}; 

\draw[->] (1) to node[above,la]{$\int_\cC$} node[below,la]{$\simeq_{QE}$} (2); 
\draw[right hook->] (1) to node[left,la]{$\iota_*$} (1'); 
\draw[right hook->] (2) to node[right,la]{$N$} (2');
\draw[->] (1') to node[below,la]{$\mathrm{Un}_\cC$} node[above,la]{$\simeq_{QE}$} (2');
\end{tz}
\end{theorem}

\begin{proof}
    This follows from \cref{prop:leftQuillencommute1} and \cite[Lemma E.2.12]{JoyalVolumeII}, since all objects in $[\cC^{\op},\Set]_\mathrm{proj}$ are fibrant. 
\end{proof}

\section{Comparison with infinity: cartesian case}\label{section:cartinfty}

This section is analogous to the preceding one, except that we focus on the case of Grothendieck fibrations. That is, we want to compare the model structure on $\Cat^+_{/\cC^\sharp}$ for cart-marked Grothendieck fibrations constructed in \cref{thm:cartesianMS} to its $\infty$-counterpart. 

Our decision to resolve our technical issues through the use of marked categories was inspired by the fact that, in the corresponding $\infty$-setting, Lurie's strategy is to use \emph{marked simplicial sets}, which are pairs $(X,E)$ of a simplicial set $X$ and a set $E$ of $1$-simplices of~$X$ containing the degenerate ones. Together with maps of simplicial sets that preserve the marking, they form a category $\sSet^+$. Similarly to \cref{forgetmarkings}, we have functors $U\colon \sSet^+\to \sSet$, $(-)^\flat\colon \sSet\to \sSet^+$, and $(-)^\sharp\colon \sSet\to \sSet^+$ given by forgetting the marking, marking only degeneracies, and marking every $1$-simplex, respectively.

In \cite[Proposition 3.1.3.7]{HTT}, Lurie constructs a model structure on the slice category $\sSet^+_{/(N\cC)^\sharp}$, called the \emph{cartesian model structure} and denoted here by $(\sSet^+_{/(N\cC)^\sharp})_\mathrm{Cart}$. This model structure has as cofibrations the monomorphisms of simplicial sets, and its fibrant objects are the \emph{naturally marked cartesian fibrations}---the $\infty$-analogue of cart-marked Grothendieck fibrations. 

The goal of this section is to study the following square of Quillen adjunctions.
\begin{diagram} \label{QP2}
    \node[](1) {$[\cC^{\op},\Cat]_{\mathrm{proj}}$}; 
    \node[below of=1,yshift=-.5cm](2) {$[\cC^\op,\sSet^+]_\mathrm{proj}$}; 
    \node[right of=1,xshift=2.8cm](3) {$(\Cat^+_{/\cC^\sharp})_{\mathrm{GrFib}}$}; 
    \node[below of=3,yshift=-.5cm](4) {$(\sSet^+_{/(N\cC)^\sharp})_{\mathrm{Cart}}$};

    \draw[->] ($(1.east)-(0,5pt)$) to node[below,la]{$\int^+_\cC$} ($(3.west)-(0,5pt)$);
\draw[->] ($(3.west)+(0,5pt)$) to node[above,la]{$\cT^+_\cC$} ($(1.east)+(0,5pt)$);
\node[la] at ($(1.east)!0.5!(3.west)$) {$\bot$};
\draw[->] ($(2.east)-(0,5pt)$) to node[below,la]{$\mathrm{Un}^+_\cC$} ($(4.west)-(0,5pt)$);
\draw[->] ($(4.west)+(0,5pt)$) to node[above,la]{$\mathrm{St}^+_\cC$} ($(2.east)+(0,5pt)$);
\node[la] at ($(2.east)!0.5!(4.west)$) {$\bot$};

\draw[->] ($(3.south)+(5pt,0)$) to node[right,la]{$N^+$} ($(4.north)+(5pt,0)$); 
\draw[->] ($(4.north)-(5pt,0)$) to node[left,la]{$c^+$} ($(3.south)-(5pt,0)$); 
\node[la] at ($(3.south)!0.5!(4.north)$) {\rotatebox{90}{$\bot$}};
\draw[->] ($(1.south)+(5pt,0)$) to node[right,la]{$(N^+(-)^\natural)_*$} ($(2.north)+(5pt,0)$); 
\draw[->] ($(2.north)-(5pt,0)$) to node[left,la]{$(\Loc\circ c^+)_*$} ($(1.south)-(5pt,0)$); 
\node[la] at ($(1.south)!0.5!(2.north)$) {\rotatebox{90}{$\bot$}};
\end{diagram}
The adjunction at the top, presented in \cref{adjunctionmarkedgroth}, has the marked Grothendieck construction as the right adjoint, and is a Quillen equivalence by \cref{thm:markedGCQE}. The bottom adjunction is referred to as \emph{straightening-unstraightening}, and \cite[Theorem 3.2.0.1]{HTT} shows it is a Quillen equivalence when considering the projective model structure on functors into $\sSet^+$ endowed with the model structure $(\sSet^+_{/\Delta[0]})_\mathrm{Cart}$. In \cref{markednerveadj,markednerveQrefl,cor:cNmarkedQP}, we construct the adjunctions on the right and left sides and show that they both are Quillen reflections. Finally, \cref{prop:leftQuillencommute,prop:rightQuillencommute} prove that the squares of left and right adjoints in \eqref{QP2} commute up to a level-wise weak equivalence.

We begin by studying the adjunction $c^+\dashv N^+$.

\begin{defn}
We define the \textbf{marked nerve} $N^+\colon \Cat^+\to \sSet^+$ to be the functor sending 
\begin{numbered}[leftmargin=1.1cm]
\item a marked category $(\cC,E)$ to the marked simplicial set $(N\cC,E)$, 
\item a marked functor $F\colon (\cC,E_\cC)\to (\cD,E_\cD)$ to the marked map of simplicial sets $NF\colon (N\cC,E_{\cC})\to (N\cD,E_{\cD})$. 
\end{numbered}
We define the \textbf{marked categorification} $c^+\colon \sSet^+\to \Cat^+$ to be the functor sending
\begin{numbered}[leftmargin=1.1cm]
\item a marked simplicial set $(X,E)$ to the marked category $(cX,E')$, where a morphism $f$ of $cX$ is in $E'$ if and only if there is a marked $1$-simplex that represents it,
\item a marked map of simplicial sets $F\colon (X,E_X)\to (Y,E_Y)$ to the marked functor $cF\colon (cX,E_X')\to (cY,E_Y')$. 
\end{numbered}
\end{defn}

\begin{prop}\label{markednerveadj}
There is an adjunction 
\begin{tz}
\node[](1) {$\Cat^+$}; 
\node[right of=1,xshift=1.2cm](2) {$\sSet^+$}; 

\draw[->] ($(1.east)-(0,5pt)$) to node[below,la]{$N^+$} ($(2.west)-(0,5pt)$);
\draw[->] ($(2.west)+(0,5pt)$) to node[above,la]{$c^+$} ($(1.east)+(0,5pt)$);

\node[la] at ($(1.east)!0.5!(2.west)$) {$\bot$};
\end{tz}
whose right adjoint is fully faithful. 
\end{prop}

\begin{proof}
As we know that $c\dashv N$ is already an adjunction, it is enough to show that, given a marked simplicial set $(X,E)$ and a marked category $(\cC,M)$, a functor $F\colon cX\to \cC$ preserves the marking if and only if its transpose $\overline{F}\colon X\to N\cC$ preserves the marking. More explicitly, if $E'$ denotes the marking in $cX$ as defined above, we need to show that $F(E')\subseteq M$ if and only if $\overline{F}(E)\subseteq M$. However, this is clear by definition. 
\end{proof}

Since $N^+(\cC^\sharp)=(N\cC)^\sharp$, the adjunction $c^+\dashv N^+$ induces an adjunction between slices over $\cC^\sharp$ and $(N\cC)^\sharp$, which is a Quillen reflection as we now show. 

\begin{prop}\label{markednerveQrefl}
The adjunction 
\begin{tz}
\node[](1) {$(\Cat^+_{/\cC^\sharp})_\mathrm{GrFib}$}; 
\node[right of=1,xshift=2.7cm](2) {$(\sSet^+_{/(N\cC)^\sharp})_\mathrm{Cart}$}; 

\draw[->] ($(1.east)-(0,5pt)$) to node[below,la]{$N^+$} ($(2.west)-(0,5pt)$);
\draw[->] ($(2.west)+(0,5pt)$) to node[above,la]{$c^+$} ($(1.east)+(0,5pt)$);

\node[la] at ($(1.east)!0.5!(2.west)$) {$\bot$};
\end{tz}
is a Quillen reflection. 
\end{prop}

\begin{proof}
First note that $c^+$ preserves cofibrations, as $c\colon \sSet\to \Cat$ is left Quillen and so takes monomorphisms to injective-on-objects functors. Following the same reasoning as in the proof of \cref{discQuillenrefl}, it suffices to show that $c^+$ preserves anodyne extensions.

As explained in \cite[Lemma 4.34]{kim}, the anodyne extensions of the model structure $(\sSet^+_{/(N\cC)^\sharp})_\mathrm{Cart}$ are cofibrantly generated by the following morphisms: 
\begin{rome}[leftmargin=1.1cm]
\item the inclusion $\Lambda^t[n]^\flat\to \Delta[n]^\flat\to (N\cC)^\sharp$ for all $n\geq 2$ and $0<t<n$, 
\item the inclusion $(\Lambda^n[n], \{n-1\to n\}\cap \Lambda^n[n]_1)\to (\Delta[n], \{n-1\to n\})\to (N\cC)^\sharp$ for all $n\geq 1$, 
\item the inclusion $(\Delta[2],\{0\to 1,1\to 2\})\to \Delta[2]^\sharp\to (N\cC)^\sharp$, 
\item the inclusion $(N\bI)^\flat\to (N\bI)^\sharp\to (N\cC)^\sharp$.
\end{rome}  When we apply $c$, we see that most of these generating morphisms are mapped to the identity, except for the following:
\begin{enumerate}[leftmargin=1.1cm]
    \item the map in (ii) for $n=1$ is sent to the anodyne extension $[0]\xrightarrow{1} [1]^\sharp\to \cC^\sharp$, 
    \item the map in (ii) for $n=2$ is sent to the anodyne extension \[ (\Lambda^2[2], \{1\to 2\})\to ([2],\{1\to 2\})\to \cC^\sharp, \] 
    \item the map in (ii) for $n=3$ is sent to a pushout of the generating anodyne extension $([2]\sqcup_{\Lambda^2[2]} [2],\{1\to 2\})\to ([2], \{1\to 2\})\to \cC^\sharp$,
    \item the map in (iii) is sent to the anodyne extension $([2],\{0\to 1,1\to 2\})\to [2]^\sharp\to \cC^\sharp$, 
    \item the map in (iv) is sent to the anodyne extension $\bI^\flat\to \bI^\sharp\to \cC^\sharp$.
\end{enumerate}
In particular, all of these maps are anodyne extensions in $(\Cat^+_{/\cC^\sharp})_{\mathrm{GrFib}}$, and so $c^+$ is left Quillen.

Moreover, since $c^+N^+\cong \id_{\Cat^+}$, we have that the component of the (derived) counit at a cart-marked Grothendieck fibration $P\colon (\cP,E)\to \cC^\sharp$ is an isomorphism, showing that $c^+\dashv N^+$ is a Quillen reflection. 
\end{proof}

Considering the particular case of $\cC=[0]$ yields the following. 

\begin{cor} \label{cor:cNmarkedQP}
    The adjunction 
\begin{tz}
\node[](1) {$\Cat^+$}; 
\node[right of=1,xshift=1.2cm](2) {$\sSet^+$}; 

\draw[->] ($(1.east)-(0,5pt)$) to node[below,la]{$N^+$} ($(2.west)-(0,5pt)$);
\draw[->] ($(2.west)+(0,5pt)$) to node[above,la]{$c^+$} ($(1.east)+(0,5pt)$);

\node[la] at ($(1.east)!0.5!(2.west)$) {$\bot$};
\end{tz}
is a Quillen reflection. 
\end{cor}

As a consequence of \cref{cor:locisoQP,cor:cNmarkedQP}, we obtain that the adjunction on the left of \eqref{QP2} is also a Quillen reflection. 

\begin{prop}
The composite of adjunctions
\begin{tz}
\node[](0) {$[\cC^{\op},\Cat]_\mathrm{proj}$};

\node[right of=0,xshift=2.6cm](1) {$[\cC^{\op},\Cat^+]_\mathrm{proj}$}; 
\node[right of=1,xshift=2.6cm](2) {$[\cC^{\op},\sSet^+]_\mathrm{proj}$}; 

\draw[->] ($(0.east)-(0,5pt)$) to node[below,la]{$(-)^\natural_*$} ($(1.west)-(0,5pt)$);
\draw[->] ($(1.west)+(0,5pt)$) to node[above,la]{$\Loc_*$} ($(0.east)+(0,5pt)$);
\node[la] at ($(0.east)!0.5!(1.west)$) {$\bot$};

\draw[->] ($(1.east)-(0,5pt)$) to node[below,la]{$N^+_*$} ($(2.west)-(0,5pt)$);
\draw[->] ($(2.west)+(0,5pt)$) to node[above,la]{$c^+_*$} ($(1.east)+(0,5pt)$);

\node[la] at ($(1.east)!0.5!(2.west)$) {$\bot$};
\end{tz}
is a Quillen reflection.
\end{prop}

Just as in the discrete case, the complexity of the straightening-unstraightening adjunction $\mathrm{St}^+_\cC\dashv \mathrm{Un}^+_\cC$ makes it hard to directly unpack the commutativity of the square in \eqref{QP2}. Hence, we first work with a simplified construction 
\begin{tz}
\node[](1) {$[\cC^{\op},\sSet^+]_\mathrm{proj}$}; 
\node[right of=1,xshift=2.8cm](2) {$(\sSet^+_{/(N\cC)^\sharp})_\mathrm{Cart}$}; 

\draw[->] ($(1.east)-(0,5pt)$) to node[below,la]{$N^+_{(-)}(\cC)$} ($(2.west)-(0,5pt)$);
\draw[->] ($(2.west)+(0,5pt)$) to node[above,la]{$\cF^+_{(-)}(\cC)$} ($(1.east)+(0,5pt)$);

\node[la] at ($(1.east)!0.5!(2.west)$) {$\bot$};
\end{tz}
given by the \emph{marked relative nerve} and its left adjoint introduced in \cite[Definition 3.2.5.12 and Remark 3.2.5.13]{HTT}. By \cite[Proposition 3.2.5.18(2)]{HTT}, this provides another Quillen equivalence between the projective and cartesian model structures. Moreover, since the left adjoint $\cF^+_{(-)}(\cC)$ is related by a level-wise weak equivalence to the straightening functor $\mathrm{St}^+_{\cC}$ by \cite[Lemma 3.2.5.17(2)]{HTT}, it is a suitable replacement to study the commutativity of the square \eqref{QP2} of adjoint functors.

\begin{prop} \label{prop:leftadjcommute}
The following square of left adjoints commutes up to isomorphism.
    \begin{tz}
\node[](1) {$\sSet^+_{/(N\cC)^\sharp}$}; 
\node[right of=1,xshift=2.1cm](2) {$[\cC^{\op},\sSet^+]$}; 
\node[below of=1](1') {$\Cat^+_{/\cC^\sharp}$};
\node[below of=2](2') {$[\cC^{\op},\Cat]$}; 

\draw[->] (1) to node[above,la]{$\cF_{(-)}^+(\cC)$} (2); 
\draw[->] (1) to node[left,la]{$c^+$} (1'); 
\draw[->] (2) to node[right,la]{$(\Loc\circ  c^+)_*$} (2');
\draw[->] (1') to node[below,la]{$\cT^+_\cC$} (2');
\end{tz}
\end{prop}

\begin{proof}
    Let us denote by $\int_{\Delta^+} (N\cC)^\sharp$ the category of elements of the marked simplicial set $(N\cC)^\sharp$. Then the slice category $\sSet^+_{/(N\cC)^\sharp}$ can be seen as a full subcategory of the presheaf category $\Set^{(\int_{\Delta^+} (N\cC)^\sharp)^{\op}}$. Hence, as every functor in the above square is a left adjoint, it is enough to show that their values agree on all representables $\Delta[n]^\flat\to (N\cC)^\sharp$ and $\Delta[1]^\sharp\to (N\cC)^\sharp$ of $\Set^{(\int_{\Delta^+} (N\cC)^\sharp)^{\op}}$. 

    First, a representable $\sigma\colon \Delta[n]^\flat\to (N\cC)^\sharp$ is sent by $c^+$ to the corresponding functor $\sigma\colon [n]^\flat\to \cC^\sharp$, which is in turn sent by $\cT^+_\cC$ to the functor $(-)\downarrow \sigma\colon \cC^{\op}\to \Cat$. On the other hand, we can express the action of $(\Loc\circ  c^+)_*\cF^+_{(-)}(\cC)$ on $\sigma$ as the composite of the left vertical map and the two bottom horizontal maps in the diagram below.

    \begin{tz}
\node[](1) {$\sSet^+_{/\Delta[n]^\sharp}$}; 
\node[right of=1,xshift=2.2cm](2) {$[[n]^{\op},\sSet^+]$}; 
\node[right of=2,xshift=2.2cm](3) {$[[n]^{\op},\Cat]$}; 
\node[below of=1](1') {$\sSet^+_{/(N\cC)^\sharp}$};
\node[below of=2](2') {$[\cC^{\op},\sSet^+]$}; 
\node[below of=3](3') {$[\cC^{\op},\Cat]$};

\draw[->] (1) to node[above,la]{$\cF_{(-)}^+([n])$} (2); 
\draw[->] (2) to node[above,la]{$(\Loc\circ c^+)_*$} (3); 
\draw[->] (1) to node[left,la]{$\sigma_!$} (1'); 
\draw[->] (2) to node[right,la]{$\sigma_!$} (2');
\draw[->] (3) to node[right,la]{$\sigma_!$} (3');
\draw[->] (1') to node[below,la]{$\cF_{(-)}^+(\cC)$} (2');
\draw[->] (2') to node[below,la]{$(\Loc\circ c^+)_*$} (3');
\end{tz}
Now, the left-hand square commutes up to isomorphism by \cite[Remark 3.2.5.14]{HTT}, and the right-hand square commutes up to isomorphism as it is straightforward to check that the corresponding square of right adjoints commutes. Hence, this allows for the following rewriting
 \[ (\Loc\circ c^+)_*\cF_{\sigma}^+(\cC)\cong \sigma_!(\Loc\circ c^+)_*\cF_{\id_{\Delta[n]}}^+([n]), \]
where $\id_{\Delta[n]}\colon \Delta[n]^\flat\to \Delta[n]^\sharp$. By \cite[Example 3.2.5.6]{HTT}, there is an isomorphism $\cF_{\id_{\Delta[n]}}^+([n])\cong N((-)\downarrow [n])^\flat$ and so we have 
\[ (\Loc\circ c^+)_*\cF_{\id_{\Delta[n]}}^+([n])\cong c N((-)\downarrow [n])\cong (-)\downarrow [n]. \] Finally, by \cref{lem:LKE}, the left Kan extension along $\sigma\colon [n]\to \cC$ is given by \[ \sigma_! ((-)\downarrow [n])\cong (-)\downarrow \sigma \] and so we get the desired result. 

Next, we turn to the representable $\sigma\colon \Delta[1]^\sharp\to (N\cC)^\sharp$. This is sent by $c^+$ to the corresponding functor $\sigma\colon [1]^\sharp\to \cC^\sharp$, which is in turn sent by $\cT^+_\cC$ to the functor $((-)\downarrow \sigma)^{\simeq}$, where $(-)^{\simeq}\colon \Cat\to \Gpd$ is the groupoidification functor (the left adjoint of the full inclusion of groupoids into categories). As before, we have an isomorphism
\[(\Loc\circ c^+)_*\cF_{\sigma}^+(\cC)\cong \sigma_!(\Loc\circ c^+)_*\cF_{\id_{\Delta[1]^\sharp}}^+([1]). \]
Using \cite[Example 3.2.5.6]{HTT} and keeping track of the markings, there is an isomorphism $\cF_{\id_{\Delta[1]^\sharp}}^+([1])\cong (N((-)\downarrow [1]),E_{(-)})$, where $E_i$ contains all $1$-simplices of $N(i\downarrow [1])$ for $i=0,1$, and so we have
\[ (\Loc\circ c^+)_*\cF_{\id_{\Delta[1]^\sharp}}^+([1])\cong (c N((-)\downarrow [1]))[E_{(-)}^{-1}] \cong ((-)\downarrow [1])^\simeq. \] 
Since post-composing with the groupoidification functor commutes with left Kan extensions and using \cref{lem:LKE}, we get that 
\[ \sigma_!(((-)\downarrow [1])^\simeq)\cong (\sigma_!((-)\downarrow [1]))^\simeq\cong ((-)\downarrow\sigma)^\simeq, \]
which gives the desired result. 
\end{proof}

We can now use \cref{prop:leftadjcommute} to obtain the commutativity of the square of left adjoints of interest.

\begin{prop} \label{prop:leftQuillencommute}
The following square of left Quillen functors commutes up to a level-wise weak equivalence.
    \begin{tz}
\node[](1) {$(\sSet^+_{/(N\cC)^\sharp})_\mathrm{Cart}$}; 
\node[right of=1,xshift=2.8cm](2) {$[\cC^{\op},\sSet^+]_\mathrm{proj}$}; 
\node[below of=1](1') {$(\Cat^+_{/\cC^\sharp})_\mathrm{GrFib}$};
\node[below of=2](2') {$[\cC^{\op},\Cat]_\mathrm{proj}$}; 

\draw[->] (1) to node[above,la]{$\mathrm{St}^+_\cC$} node[below,la]{$\simeq_{QE}$} (2); 
\draw[->] (1) to node[left,la]{$c^+$} (1'); 
\draw[->] (2) to node[right,la]{$(\Loc\circ  c^+)_*$} (2');
\draw[->] (1') to node[below,la]{$\cT^+_\cC$} node[above,la]{$\simeq_{QE}$} (2');
\end{tz}
\end{prop}

\begin{proof}
    By \cite[Lemma 3.2.5.17(2)]{HTT}, there is a natural transformation 
    \[ \mathrm{St}^+_\cC \Rightarrow \cF^+_{(-)}(\cC) \]
    whose components are weak equivalences in $[\cC^{\op},\sSet^+]_\mathrm{proj}$. Now, note that the composite of left Quillen functors 
    \[ \sSet^+\xrightarrow{c^+} \Cat^+ \xrightarrow{\Loc} \Cat\]
    preserves all weak equivalences since all objects are cofibrant in $\sSet^+$. Hence the functor 
    \[ (\Loc\circ  c^+)_*\colon [\cC^{\op},\sSet^+]_\mathrm{proj}\to [\cC^{\op},\Cat]_\mathrm{proj}\]
    also preserves all weak equivalences as they are defined level-wise. This implies that the natural transformation
    \[ (\Loc\circ  c^+)_*\mathrm{St}^+_\cC \Rightarrow (\Loc\circ  c^+)_*\cF^+_{(-)}(\cC) \]
     is level-wise a weak equivalence in $[\cC^{\op},\Cat]_\mathrm{proj}$. Combining with the isomorphism from \cref{prop:leftadjcommute}, we get the desired result. 
\end{proof}

\begin{theorem}\label{prop:rightQuillencommute}
    The following square of homotopically fully faithful right Quillen functors commutes up to a level-wise weak equivalence. 
    \begin{tz}
\node[](1) {$[\cC^{\op},\Cat]_{\mathrm{proj}}$}; 
\node[right of=1,xshift=2.8cm](2) {$(\Cat^+_{/\cC^\sharp})_{\mathrm{GrFib}}$}; 
\node[below of=1](1') {$[\cC^\op,\sSet^+]_\mathrm{proj}$};
\node[below of=2](2') {$(\sSet^+_{/(N\cC)^\sharp})_{\mathrm{Cart}}$}; 

\draw[->] (1) to node[above,la]{$\int^+_\cC$} node[below,la]{$\simeq_{QE}$} (2); 
\draw[right hook->] (1) to node[left,la]{$(N^+(-)^\natural)_*$} (1'); 
\draw[right hook->] (2) to node[right,la]{$N^+$} (2');
\draw[->] (1') to node[below,la]{$\mathrm{Un}_\cC^+$} node[above,la]{$\simeq_{QE}$} (2');
\end{tz}
\end{theorem}

\begin{proof}
    This follows from \cref{prop:leftQuillencommute} and \cite[Lemma E.2.12]{JoyalVolumeII}, since all objects in $[\cC^{\op},\Cat]_\mathrm{proj}$ are fibrant. 
\end{proof}

\bibliographystyle{alpha}
\bibliography{references}

\end{document}